\documentclass[a4paper,11pt]{article}
\usepackage{makeidx}
\usepackage[mathscr]{eucal}
\usepackage[dvipdf]{color} 
\usepackage{amssymb, amsfonts, amsmath, amsthm, graphics, comment}

\newtheorem{proposition}{Proposition}[section]
\newtheorem{theorem}{Theorem}[section]
\newtheorem{lemma}{Lemma}[section]
\newtheorem{remark}{Remark}[section]

\newtheorem{notation}{Notation}[section]
\newtheorem{assumption}{Assumption}[section]

\numberwithin{equation}{section}

\title{Nondegeneracy of ground states for nonlinear scalar field equations 
 involving the Sobolev-critical exponent at high frequencies
 in three and four dimensions}

\author{Takafumi Akahori, Miho Murata}
\date{}
\newcommand{\Rd}{\mathbb{R}^d}

\begin{document}
\maketitle

\begin{abstract}
We consider nonlinear scalar field equations involving the Sobolev-critical exponent at high frequencies $\omega$. Since the limiting profile of the ground state as $\omega\to \infty$ is the Aubin-Talenti function and degenerate in a certain sense,  from the point of view of perturbation methods, 
 the nondegeneracy problem for the ground states at high frequencies is subtle.  In addition, since the limiting profile (Aubin-Talenti function) fails to lie in $L^{2}(\mathbb{R}^{d})$ for $d=3,4$, the nondegeneracy problem for $d=3,4$ is more difficult than that for $d\ge 5$ and an applicable methodology is not known. In this paper, we solve the nondegeneracy problem for $d=3,4$  by modifying the arguments in \cite{AIK, AIIKN}.
 We also show that the linearized operator around the ground state has exactly one negative eigenvalue.
\end{abstract}


\noindent 
{\bf Mathematics Subject Classification}~35J20, 35B09, 35B33, 35Q55



\section{Introduction}\label{21/10/14/11:19}

We consider the following nonlinear scalar field equation involving the Sobolev critical exponent:   
\begin{equation}\label{eq:1.1} 
 - \Delta u + \omega u  
-|u|^{\frac{4}{d-2}}u 
-
g(u)
=0 , 
\qquad u \colon \mathbb{R}^{d} \to \mathbb{C}
,
\end{equation}
where $d\ge 3$, $\omega>0$, and $g$ is a complex-valued function on $\mathbb{C}$ to be specified shortly (see Assumption \ref{21/12/23/14:45}). It is worthwhile noting that the existence of the term $g$ is essential; Indeed, the Pohozaev's identity (see \eqref{18/11/24/17:41}) shows that if $g$ is not placed, then \eqref{eq:1.1} has no solution in $H^{1}(\mathbb{R}^{d})$ for all $\omega>0$. 

Our aim is to prove the nondegeneracy of ground states to \eqref{eq:1.1} under   certain general assumptions about the function $g$ (see Theorem \ref{20/8/17/11:39}). Moreover, we show that the linearized operator around the ground state  has exactly one negative eigenvalue (see Theorem \ref{21/9/17/9:1}). 
 These results play important roles in the study of dynamics for the corresponding time-dependent nonlinear Schr\"odinger equation (see \cite{AIKN3, Nakanishi-Schlag2}). Moreover, the nondegeneracy implies that there is no bifurcation branch from the ground state (see \cite{Crandall-Rabinowitz}). 

Although the nondegeneracy of ground states has been studied for a large class of nonlinear elliptic equations, we have not found a methodology completely applicable to the equation \eqref{eq:1.1} at high frequencies ($\omega \gg 1$) for $d=3,4$, even for the typical model $g(u)=|u|^{p-1}u$ with $1<p<\frac{d+2}{d-2}$ (see Remark \ref{22/3/8/18:16}).
 The main difficulty is that the limiting profile of ground states as $\omega \to \infty$ is degenerate for all $d\ge 3$ and fails to lie in $L^{2}(\mathbb{R}^{d})$ for $d=3,4$; Precisely, by an appropriate scaling, the equation \eqref{eq:1.1} at high frequencies ($\omega \gg 1$) can be thought of as a perturbation of the equation  
\begin{equation}\label{21/10/16/11:41}
-\Delta u - |u|^{\frac{4}{d-2}} u=0, 
\qquad u \colon \mathbb{R}^{d}\to \mathbb{C}. 
\end{equation}
Note that the equation \eqref{21/10/16/11:41} is invariant under the scaling $u(x)\mapsto \lambda^{-1}u(\lambda^{-\frac{2}{d-2}}x)$. 
 It is known that the ``ground state'' to \eqref{21/10/16/11:41}
 is the Aubin-Talenti function  
\begin{equation}\label{eq:1.16}
W(x)
:=\Big(1+\frac{|x|^{2}}{d(d-2)} \Big)^{-\frac{d-2}{2}}
.
\end{equation}
Hence, we can expect that the limiting profile of ground states to \eqref{eq:1.1} as $\omega\to \infty$ is the Aubin-Talenti function $W$; Actually, we can prove this (see Lemma \ref{theorem:3.1}) and therefore see that the difficulty in  the nondegeneracy problem arises from $W$ being degenerate in a certain sense for all $d\ge 3$ (see \eqref{20/9/13/9:52}) and not in $L^{2}(\mathbb{R}^{d})$ for $d=3,4$.

Now, we state the assumptions about the function $g$:   
\begin{assumption}\label{21/12/23/14:45}  
The function $g\colon \mathbb{C}\to \mathbb{C}$ in \eqref{eq:1.1} is independent of the frequency $\omega$ and satisfies the following conditions: 
\begin{enumerate}
\item
The restriction of $g$ on the open interval $(0,\infty)$ is a real-valued $C^{2}$-function, and  
\begin{equation}\label{21/11/28/11:21}
g(0)=0,
\qquad 
g(z)=\dfrac{z}{|z|}g(|z|)
\quad 
\mbox{for all $z\in \mathbb{C}\setminus \{0\}$}
.
\end{equation}
Furthermore, $g$ satisfies the following condition:  
\begin{equation}\label{22/3/6/12:26}
g'(t)t -g(t) \ge 0
\qquad 
\mbox{for all $t> 0$}.
\end{equation}
We remark that the condition \eqref{22/3/6/12:26} is  used mainly to ensure the existence of ground states to \eqref{eq:1.1}.


\item 
For $d\ge 3$, there exist $\max\{\frac{2}{d-2}, 1\}<p_{1} <\frac{d+2}{d-2}$,
 $\max\{\frac{d+2}{d-2}-2, p_{1}\}<p_{2}<\frac{d+2}{d-2}$ and $C_{1},C_{2}>0$ such that   
\begin{equation}\label{21/12/23/15:11}
\limsup_{t\to 0}\frac{|g'(t)|}{t^{p_{1}-1}} \le C_{1},
\qquad 
\lim_{t\to \infty}\frac{g'(t)}{t^{p_{2}-1}}=C_{2}
.
\end{equation}
Note that the assumption about $p_{1}$ and $p_{2}$ reads: if $d=3$, then $2< p_{1}< p_{2}$ and $3<p_{2}<5$; and if $d\ge 4$, then $1<p_{1}< p_{2}<\frac{d+2}{d-2}$. Moreover, the second condition in \eqref{21/12/23/15:11} implies 
\begin{equation}\label{21/12/25/15:11}
\lim_{t\to \infty}
\Big| \frac{g(t)}{t^{p_{2}}} -\frac{C_{2}}{p_{2}} \Big|=0
.
\end{equation}
 

\item Define 
\begin{equation}\label{21/12/26/17:42} 
2^{*}:=\frac{2d}{d-2}, 
\qquad 
G(t):=\int_{0}^{t}g(s)\,ds 
.
\end{equation}
Then, for $d\ge 3$, there exists $C_{3}>0$ such that 
\begin{equation}\label{22/3/8/11:29}
G(t)-\frac{1}{2^{*}}g(t)t \ge C_{3} t^{p_{2}+1}  
\qquad 
\mbox{for all $t\ge 0$}
,
\end{equation}
where $p_{2}$ is the same constant as in \eqref{21/12/23/15:11}. 
 We remark that the function on the left-hand side of \eqref{22/3/8/11:29} appears in the Pohozaev's identity (see \eqref{18/11/24/17:41}). 
\end{enumerate} 
\end{assumption}

We give some remarks about Assumption \ref{21/12/23/14:45}:
\begin{remark}\label{22/1/9/9:1}
\begin{enumerate}
\item 
The assumption $p_{2}>\frac{d+2}{d-2}-2$ is related to the existence of ground states; In fact, it is known (see \cite{AIKN6,Wei-Wu}) that when $d=3$, $1<p\le 3$ and $g(u)=|u|^{p-1}u$, there exists $\omega_{c}>0$ such that if $\omega > \omega_{c}$, then there is no ground state to \eqref{eq:1.1}. 


\item 
The condition \eqref{22/3/6/12:26} implies 
\begin{equation}\label{22/3/5/17:35}
g(t)t - 2 G(t) \ge 0
\qquad 
\mbox{for all $t> 0$}.
\end{equation}
Furthermore, it follows from  \eqref{22/3/8/11:29} and \eqref{22/3/5/17:35} 
 that 
\begin{equation}\label{22/3/3/8:40}
\Big( \frac{1}{2}-\frac{1}{2^{*}} \Big)g(t)
\ge 
\frac{1}{t}\big\{ G(t)-\frac{1}{2^{*}}g(t)t \big\}
\ge C_{3} t^{p_{2}} > 0 
\qquad 
\mbox{for all $t> 0$}
.
\end{equation}
In particular, $g$ is positive on $(0,\infty)$. 

\item 
The assumption \eqref{21/12/23/15:11} together with $g(0)=0$ shows 
 that the following hold for all $t\ge 0$:
\begin{equation}\label{21/12/24/13:39}
|g'(t)| \lesssim t^{p_{1}-1} + t^{p_{2}-1},
\quad 
|g(t)| \lesssim t^{p_{1}}+t^{p_{2}},
\quad 
|G(t)| \lesssim t^{p_{1}+1}+t^{p_{2}+1}
,
\end{equation}
where the implicit constants depend only on $p_{1}$, $p_{2}$, $C_{1}$ 
 and $C_{2}$.  


\item 
Let $d\ge 3$, $\max\{\frac{2}{d-2},1\} <p_{1} <  \frac{d+2}{d-2}$, $\max\{\frac{d+2}{d-2}-2, p_{1}\}<p_{2} <\frac{d+2}{d-2}$, and $p_{1}<q<p_{2}$. 
 Then, typical models of the function $g$ satisfying 
 all conditions in Assumption \ref{21/12/23/14:45} are the following:
\begin{align}
\label{ex1}
g(t)&=t^{p_{2}}, 
\qquad 
g(t)=t^{p_{1}} + t^{p_{2}}, 
\qquad 
g(t)=t^{p_{1}}+t^{q} + t^{p_{2}},
\\[6pt]
\label{ex2}
g(t)&=
t^{p_{1}}- \gamma t^{q} + t^{p_{2}} 
\quad 
\mbox{with a small $\gamma>0$ depending on $d$, $p_{1}$, $p_{2}$ and $q$}
.
\end{align}
\end{enumerate} 
\end{remark}


Next, we make it clear the meaning of ground state. In this paper, by a ground state to \eqref{eq:1.1}, we mean a minimizer of the following problem: 
\begin{equation}\label{18/12/07/09:20}
m_{\omega}
:=\inf\{ 
\mathcal{S}_{\omega}(u) \colon  
u \in H^{1}(\mathbb{R}^{d})\setminus \{0\},~\mathcal{N}_{\omega}(u)=0  
\},
\end{equation}
where $\mathcal{S}_{\omega}$ and $\mathcal{N}_{\omega}$ denote the action and the Nehari functional associated with \eqref{eq:1.1}, namely, 
\begin{align}
\label{eq:1.5}
\mathcal{S}_{\omega}(u)
&:=
\frac{1}{2}\|\nabla u \|_{L^{2}}^{2}
+
\frac{\omega}{2}\|u\|_{L^{2}}^{2}
-
\frac{1}{2^{*}}\|u\|_{L^{2^{*}}}^{2^{*}}
-
\int_{\mathbb{R}^{d}}G(|u|)
,
\\[6pt]
\label{20/8/5/10:4}
\mathcal{N}_{\omega}(u)
&:=
\mathcal{S}_{\omega}'(u) u
=
\|\nabla u \|_{L^{2}}^{2}
+
\omega \|u\|_{L^{2}}^{2}
-
\|u\|_{L^{2^{*}}}^{2^{*}}
-
\int_{\mathbb{R}^{d}}g(|u|)|u| 
.
\end{align}


Under Assumption \ref{21/12/23/14:45}, we can prove the existence of ground states to \eqref{eq:1.1} (see, e.g., Section A of \cite{AIIKN}, Proposition 2.0.1 of \cite{AIKN3} and Section 4.2 of \cite{Willem}): 
\begin{proposition}\label{21/11/26/15:35}  
Assume $d\ge 3$. Then, under Assumption \ref{21/12/23/14:45}, for any $\omega>0$, there exists a ground state $\Phi_{\omega}$ to \eqref{eq:1.1} with the following properties: $\Phi_{\omega}$ is a solution to \eqref{eq:1.1}; 
 $\Phi_{\omega} \in H^{2}(\mathbb{R}^{d})\cap C^{2}(\mathbb{R}^{d})$;  
 $\Phi_{\omega}$ is positive,  radially symmetric about $0$, and strictly decreasing as a function of $|x|$ (hence $\|\Phi_{\omega}\|_{L^{\infty}}=\Phi_{\omega}(0)$).  
\end{proposition}

In order to describe our main results (Theorem \ref{20/8/17/11:39} and 
 Theorem \ref{21/9/17/9:1}), we  introduce several symbols: 
\begin{notation}\label{21/12/23/15:24}
Let $d\ge 3$. 
\begin{enumerate}
\item 
For $\omega>0$, we use $\Phi_{\omega}$ to denote a ground state to \eqref{eq:1.1} given in Proposition \ref{21/11/26/15:35} (we do not need the uniqueness of ground states in this paper).


\item 
We use $H_{\rm{rad}}^{1}(\mathbb{R}^{d})$ to denote the set of functions in $H^{1}(\mathbb{R}^{d})$ which is radially symmetric about $0$. 


\item 
For $f,g \in L^{2}(\mathbb{R}^{d})$, we use $\langle f, g\rangle$ to denote the  inner product in the real Hilbert space $L_{\rm{real}}^{2}(\mathbb{R}^{d})$, namely,   
\begin{equation}\label{21/12/23/15:29}
\langle f, g \rangle :=\Re\int_{\mathbb{R}^{d}}f(x) \overline{g(x)}\,dx 
.
\end{equation} 


\item 
For $\omega>0$, we use $L_{\omega,+}$ to denote (the real-part of) the linearized operator around $\Phi_{\omega}$, namely,  
\begin{equation}\label{18/12/07/10:05}
L_{\omega,+} 
:=  
-\Delta  + V_{\omega,+}
\qquad 
\mbox{with}~~
V_{\omega,+}
:=
\omega - 
\frac{d+2}{d-2} \Phi_{\omega}^{\frac{4}{d-2}}
-g'(\Phi_{\omega}) 
.
\end{equation}
Note that $L_{\omega,+}$ is a self-adjoint operator on the real Hilbert space $L_{\rm{real}}^{2}(\mathbb{R}^{d})$ with domain $H^{2}(\mathbb{R}^{d})$. Furthermore, by a standard argument in the spectral theory, we see that 
 the essential spectrum $\sigma_{\rm{ess}}(L_{\omega,+})$ is identical to $[\omega, \infty)$, so that the discrete spectrum is included in $(-\infty, \omega)$:
\begin{equation}\label{22/3/22/11:4}
\sigma_{\rm{disc}}(L_{\omega,+}) \subset (-\infty, \omega). 
\end{equation}

\item 
We use $B_{\omega} \colon H^{1}(\mathbb{R}^{d})\times H^{1}(\mathbb{R}^{d})\to \mathbb{R}$ to denote the bilinear form associated with $L_{\omega,+}$, namely,  
\begin{equation}\label{21/10/9/9:29} 
B_{\omega}(u,v)
:=
\langle \nabla u,   \nabla v \rangle
+
\langle V_{\omega,+} u, v \rangle
.
\end{equation}

\end{enumerate}  
\end{notation}

\begin{remark}\label{21/10/7/10:57}
\begin{enumerate}
\item 
We can verify that the action $\mathcal{S}_{\omega}$ and the Nehari functional $\mathcal{N}_{\omega}$ (see \eqref{eq:1.5} and \eqref{20/8/5/10:4}) are $C^{2}$ on the real Hilbert space $H_{\rm{real}}^{1}(\mathbb{R}^{d})$, and   
\begin{align}
\label{21/12/23/15:59}
\mathcal{S}_{\omega}'(u) v
&=
\langle -\Delta u  + \omega u  
-|u|^{\frac{4}{d-2}}u-g(u), v  
\rangle 
\quad 
\mbox{for all $u,v \in H^{1}(\mathbb{R}^{d})$}
,
\\[6pt]
\label{21/9/16/14:33}
[\mathcal{S}_{\omega}''(\Phi_{\omega})u ]v
&=
B_{\omega}(u,v)
\quad 
\mbox{for all $u,v \in H^{1}(\mathbb{R}^{d})$}
,
\\[6pt]
\label{21/9/15/16:23}
\mathcal{N}_{\omega}'(\Phi_{\omega})u
&=
\langle L_{\omega,+} \Phi_{\omega} 
, u
\rangle
\quad 
\mbox{for all $u \in H^{1}(\mathbb{R}^{d})$}
.
\end{align}
\item 
Since $\Phi_{\omega}$ is a solution to \eqref{eq:1.1}, 
 differentiating the equation, we see that 
 the derivatives of $\Phi_{\omega}$ belong to the kernel of $L_{\omega,+}$: 
\begin{equation}\label{21/9/18/10:40}
\partial_{1}\Phi_{\omega}, \ldots, \partial_{d}\Phi_{\omega} 
\in  \operatorname{Ker}{L_{\omega,+}}
.
\end{equation}
\end{enumerate}
\end{remark}

Now, we state the main results of this paper. The first one is concerning 
 the nondegeneracy of ground states at high frequencies:
\begin{theorem}\label{20/8/17/11:39}
Assume $d=3,4$. Under Assumption \ref{21/12/23/14:45} and Notation \ref{21/12/23/15:24}, there exists $\omega_{\rm{ng}}>0$ such that the following hold for all  $\omega>\omega_{\rm{ng}}$:
\begin{enumerate}
\item[(i)]  The ground state $\Phi_{\omega}$ to \eqref{eq:1.1} is nondegenerate in $H_{\rm{rad}}^{1}(\mathbb{R}^{d})$; namely, if  $u \in H_{\rm{rad}}^{1}(\mathbb{R}^{d})$ satisfies $B_{\omega}(u,v)=0$ for all $v \in H_{\rm{rad}}^{1}(\mathbb{R}^{d})$, then $u$ must be trivial ($u\equiv 0$). 

\item[(ii)]
The kernel of $L_{\omega,+}$ consists of the derivatives of $\Phi_{\omega}$, namely,    
\begin{equation}\label{20/8/5/10:52}
\operatorname{Ker}{L_{\omega,+}}
=
\operatorname{span}{
\{\partial_{1}\Phi_{\omega},\ldots, \partial_{d}\Phi_{\omega} \}
}.
\end{equation}
\end{enumerate} 
\end{theorem}
\begin{remark}\label{22/3/8/18:16}
When $d=3$ and $g(u)=|u|^{p-1}u$ with $3<p<5$, the nondegeneracy of ground states to \eqref{eq:1.1} had been announced in \cite{AIK}. However, there is a flaw in the proof. In this paper, we  modify the argument in \cite{AIK} and give a complete proof.  
\end{remark} 
\begin{remark}\label{22/3/6/17:50}
It is certain that the same argument as in \cite{AIIKN} proves 
 the nondegeneracy of ground sates to \eqref{eq:1.1} at high frequencies ($\omega\gg 1$) for all $d \ge 5$. 
\end{remark}

The second result is concerning the negative eigenvalues of $L_{\omega,+}$:  
\begin{theorem}\label{21/9/17/9:1}
Assume $d\ge 3$. Then, under Assumption \ref{21/12/23/14:45} and Notation \ref{21/12/23/15:24},  the following hold for all $\omega>0$:  
\begin{enumerate}
\item[(i)] 
Let $u \in H^{1}(\mathbb{R}^{d})$. 
 If $\mathcal{N}_{\omega}'(\Phi_{\omega})u=0$, then $B_{\omega}(u,u)\ge 0$.

\item[(ii)] 
$L_{\omega,+}$ has exactly one negative eigenvalue on the real Hilbert space $L_{\rm{real}}^{2}(\mathbb{R}^{d})$. Furthermore,  the multiplicity of the negative eigenvalue is one.  
\end{enumerate} 
\end{theorem}

The rest of the paper is organized as follows. 
 In Section \ref{20/8/18/11:39}, we introduce the notation besides Notation \ref{21/12/23/15:24}. 
 In Section \ref{18/11/05/09:33}, we give estimates for free and perturbed resolvents. 
 In Section \ref{18/11/05/10:01}, we give properties of ground states to \eqref{eq:1.1}.  
 In Section \ref{18/09/09/17:30}, we give a proof of Theorem \ref{20/8/17/11:39} . 
 In Section \ref{21/10/7/10:45}, we give a proof of Theorem \ref{21/9/17/9:1}. 

\subsection*{Acknowledgements}
The authors would like to thank Professor Hiroaki Kikuchi for  helpful discussion. This work was supported by JSPS KAKENHI Grant Number 20K03697. 


\section{Preliminaries}\label{20/8/18/11:39}

In this section, besides Notation \ref{21/12/23/15:24}, 
 we introduce symbols used in this paper, with auxiliary results:  
\begin{notation}\label{21/12/29/11:17}
\begin{enumerate}

\item 
Let $X$ and $Y$ be positive quantities.   
 Then, we use $X \lesssim Y$ and $X \ll Y$ to indicate the inequalities    
 $X \le C Y$ and $CX \le Y$, respectively, where $C>1$ is some constant independent of $\omega$; Note that the implicit constants in these notation are allowed to depend on the dimension $d$ and the constants $p_{1}$, $p_{2}$, $C_{1}$, $C_{2}$ and $C_{3}$ given in Assumption \ref{21/12/23/14:45}. 
 Furthermore, we use $X \sim Y$ to indicate that $X\lesssim Y \lesssim X$.


\item 
For $d\ge 3$, we use $\Lambda W$ to denote the function defined as  
\begin{equation}\label{eq:1.19}
\Lambda W
:= 
\frac{d-2}{2}W+x\cdot \nabla W
.
\end{equation}
It is worthwhile noting that the Aubin-Talenti function $W$ (see \eqref{eq:1.16}) is degenerate in the sense that  
\begin{equation}\label{20/9/13/9:52}
(-\Delta + V )\Lambda W =0 
\qquad 
\mbox{with}~~V:=-\frac{d+2}{d-2}W^{\frac{4}{d-2}}
.
\end{equation} 
By a computation involving integration by parts, we can verify that the following holds for all $d\ge 3$ and $\max\{ 1, \frac{2}{d-2}\} <r\le \frac{d+2}{d-2}$:
\begin{equation}\label{20/12/22/14:18}
\langle W^{r}, \Lambda W \rangle 
=
\frac{-\{ 4-(d-2)(r-1)\}}{2(r+1)}\|W\|_{L^{r+1}}^{r+1} 
.  
\end{equation} 


\item
For $d\ge 1$, we define the Fourier and inverse Fourier transformations to be that 
 for $u\in L^{1}(\mathbb{R}^{d})$,    
\begin{equation}\label{20/8/9/15:35}
\mathcal{F}[u](\xi):=\int_{\mathbb{R}^{d}}e^{-i \xi \cdot x}u(x)\,dx ,
\qquad 
\mathcal{F}^{-1}[u](\xi):=\frac{1}{(2\pi )^{d}}\int_{\mathbb{R}^{d}}e^{i \xi \cdot x}u(x)\,dx
.
\end{equation}


\item 
For $d\ge 3$ and $\lambda>0$, we define the $H^{1}$-scaling operator $T_{\lambda}$ by 
\begin{equation}\label{20/9/25/7:1}
T_{\lambda}[v](x):=\lambda^{-1}v(\lambda^{-\frac{2}{d-2}}x)
.
\end{equation}
Observe from the fundamental theorem of calculus that for any $\mu>0$, 
\begin{equation}\label{20/9/25/11:45}
T_{\mu}[W] - W 
=
\int_{1}^{\mu} \frac{d}{d\lambda}T_{\lambda}[W]\,d\lambda
=
-\frac{2}{d-2}
\int_{1}^{\mu} 
\lambda^{-1} 
T_{\lambda}[ \Lambda W]
\,d\lambda
. 
\end{equation}


\item 
For $d\ge 3$ and $\omega>0$, we define $M_{\omega}$ and $\widetilde{\Phi}_{\omega}$ as 
\begin{equation}\label{21/24/14:29}
M_{\omega}:=\|\Phi_{\omega}\|_{L^{\infty}}=\Phi_{\omega}(0),
\qquad 
\widetilde{\Phi}_{\omega}
:=
T_{M_{\omega}}[\Phi_{\omega}]
=
M_{\omega}^{-1}\Phi_{\omega}(M_{\omega}^{-\frac{2}{d-2}}x)
.
\end{equation}
It is worthwhile noting that 
\begin{align}
\label{21/12/24/14:35}
&\| \widetilde{\Phi}_{\omega}\|_{L^{\infty}}=\widetilde{\Phi}_{\omega}(0)=1
,
\\[6pt]
\label{22/3/10/16:12}
&-
\Delta \widetilde{\Phi}_{\omega}
+
M_{\omega}^{-\frac{4}{d-2}}\omega 
\widetilde{\Phi}_{\omega}
-
\widetilde{\Phi}_{\omega}^{\frac{d+2}{d-2}}
-
M_{\omega}^{-\frac{d+2}{d-2}}
g( M_{\omega} \widetilde{\Phi}_{\omega})
=
0
.
\end{align}


\item 
We define the function $\delta$ on $(0,\infty)$ by 
\begin{equation}\label{19/01/27/16:58}
\delta(s)
:=\left\{ \begin{array}{cc}
s^{\frac{1}{2}} & \mbox{if $d=3$},
\\[6pt]
\displaystyle{\frac{1}{\log{(1+s^{-1})}}} &\mbox{if $d=4$}.
\end{array} \right. 
\end{equation}
It is worthwhile noting (see (2.14) of \cite{AM}) that for $d=3,4$, there exists $\mathfrak{C}>0$ depending only on $d$ such that  
\begin{equation}\label{20/12/12/12:3}
\Big|
\int_{|\xi|\le 1} 
\frac{1}{(|\xi|^{2}+s)|\xi|^{2}} 
\,d\xi - \mathfrak{C} \delta(s)^{-1}
\Big|
\lesssim 1
\qquad 
\mbox{for all $0<s<1$}
.
\end{equation}


\item
We define the function $\beta$ on $(0,\infty)$ by 
\begin{equation}\label{19/01/15/08:25}
\beta(s):=\delta(s)^{-1}s
.
\end{equation}

Note that $\beta$ is strictly increasing on $(0,\infty)$, so that the inverse exists.


\item 
We use $\alpha$ to denote the inverse function of $\beta$. 
 When $d=3$, $\beta(s) =s^{\frac{1}{2}}$ and the domain of $\alpha$ is $(0,\infty)$. When $d=4$,  the image of $(0,\infty)$ by $\beta$ is $(0,1)$ (see \cite{AM}) and the domain of $\alpha$ is $(0,1)$.   
 Furthermore,  we can verify the following (see \cite{AM}): 
\begin{equation}\label{20/10/15/16:46}
\alpha(t)
\left\{ 
\begin{array}{ll}
=t^{2} & \mbox{for $d=3$ and $0< t <\infty$}, 
\\[6pt]
\sim 
\delta(t) t & \mbox{for $d=4$ and $0< t \le  T_{0}$}
, 
\end{array}
\right.  
\end{equation}
where $T_{0}>0$ is some constant.


\item 
We use $S^{d-1}$ to denote the sphere in $\mathbb{R}^{d}$ of radius $1$ centered at $0$, namely $S^{d-1}:=\{x\in \mathbb{R}^{d}\colon |x|=1\}$. Furthermore, $|S^{d-1}|$ denotes the area of $S^{d-1}$.   


\item 
For $d\ge 1$ and $q\ge 1$, we use $L_{\rm{rad}}^{q}(\mathbb{R}^{d})$ to denote the set of functions in $L^{q}(\mathbb{R}^{d})$ which is radially symmetric about $0$. 


\item 
For $d\ge 1$ and $q\ge 1$,  the symbol $L_{\rm{weak}}^{q}(\mathbb{R}^{d})$ denotes the weak $L^{q}$ space (Marcinkiewicz space). 

\end{enumerate}
\end{notation}


\section{Estimates for free and perturbed resolvents}
\label{18/11/05/09:33}

In this section, we give estimates for free and perturbed resolvents. 
  
Let us begin by recalling several estimates given in \cite{AM, CG}: 
\begin{proposition}[Proposition 1.2 of \cite{AM}]\label{18/11/17/07:17}
Assume $d= 3,4$. Let $\frac{d}{d-2} < q < \infty$.  
If $s>0$ is sufficiently small dependently on $d$ and $q$, 
 then  the inverse of the operator $1+(-\Delta + s)^{-1}V \colon L_{\rm rad}^{q}(\mathbb{R}^{d})\to L_{\rm rad}^{q}(\mathbb{R}^{d})$ exists;
 and the following estimates hold:  
\begin{enumerate} 
\item
If $f \in L_{\rm rad}^{q}(\mathbb{R}^{d})$, then  
\begin{equation}\label{19/02/02/11:47}
\|\{ 1+(-\Delta + s)^{-1}V\}^{-1} f \|_{L^{q}}
\lesssim 
\delta(s) s^{-1} 
\| f \|_{L^{q}}
,
\end{equation} 
where the implicit constant depends only on $d$ and $q$.  
\item 
If $f\in L_{\rm{rad}}^{q}(\mathbb{R}^{d})$ and $\langle f, V \Lambda W \rangle =0$, then 
\begin{equation}\label{18/11/11/15:50}
\|\{ 1+(-\Delta + s )^{-1}V\}^{-1} f \|_{L^{q}}
\lesssim 
\| f \|_{L^{q}}
,
\end{equation} 
where the implicit constant depends only on $d$ and $q$. 
\end{enumerate}
\end{proposition}


\begin{lemma}\label{18/11/05/10:29}
Assume $d \ge 1$. 
\begin{enumerate}
\item 
Let $1\le q_{1} \le q_{2} \le \infty$ 
 and $d(\frac{1}{q_{1}}-\frac{1}{q_{2}})<2$. 
 Then, the following holds for all $s>0$: 
\begin{equation}\label{18/11/05/10:41}
\|(-\Delta+s)^{-1} \|_{L^{q_{1}}(\mathbb{R}^{d}) \to L^{q_{2}}(\mathbb{R}^{d})}
\lesssim 
s^{\frac{d}{2}(\frac{1}{q_{1}}-\frac{1}{q_{2}})-1}
,
\end{equation}
where the implicit constant depends only on $d$, $q_{1}$ and $q_{2}$. 
\item 
Let $1< q_{1} \le q_{2} < \infty$ and $d(\frac{1}{q_{1}}-\frac{1}{q_{2}})<2$.
 Then, the following holds for all $s>0$: 
\begin{equation}\label{18/11/24/13:56}
\|(-\Delta+s)^{-1} \|_{L_{\rm weak}^{q_{1}}(\mathbb{R}^{d}) \to L^{q_{2}}(\mathbb{R}^{d})}
\lesssim 
s^{\frac{d}{2}(\frac{1}{q_{1}}-\frac{1}{q_{2}})-1}
,
\end{equation}
where the implicit constant depends only on $d$, $q_{1}$ and $q_{2}$.
\end{enumerate} 
\end{lemma}

\begin{lemma}\label{18/11/23/17:17}
Assume $d \ge 3$. Let $\frac{d}{d-2} < q < \infty$. 
Then, the following holds for all $s_{0} \ge 0$:   
\begin{equation}\label{18/11/23/17:19}
\|(-\Delta+s_{0})^{-1} \|_{L^{\frac{dq}{d+2q}}(\mathbb{R}^{d}) \to L^{q}(\mathbb{R}^{d})}
\lesssim 
1
,
\end{equation}
where the implicit constant depends only on $d$ and $q$. 
\end{lemma}

\begin{lemma}[Lemma 3.9 of \cite{AM}]\label{21/1/10/15:37}
Assume $d=3,4$. Then, the following holds for all $s>0$: 
\begin{equation}\label{21/1/10/7:15}
\| (-\Delta +s)^{-1} \Lambda W \|_{L^{\infty}}
\lesssim 
1+\delta(s)^{-1}
.
\end{equation}
\end{lemma}



\begin{lemma}[Lemma 3.10 of \cite{AM}]\label{18/12/02/11:25}
Assume $d=3,4$. Then, there exists a constant $\mathscr{C}>0$ depending only on $d$ such that the following holds  for all $\frac{d}{d-2}<q <\infty$,   $f \in L_{\rm rad}^{q}(\mathbb{R}^{d})$ and $s>0$: 
\begin{equation}\label{18/12/02/11:26}
\big|
\langle (-\Delta+s)^{-1} Vf , \Lambda W \rangle
+ 
\mathscr{C}
\Re\mathcal{F}[Vf](0) \delta(s)^{-1} 
\big|
\lesssim 
 \|  f  \|_{L^{q}}
,
\end{equation}
where the implicit constant depends only on $d$ and $q$.
\end{lemma}


Next, we give a variant of Lemma \ref{18/12/02/11:25}: 
\begin{lemma}\label{22/1/25/7:1}
Assume $d=3,4$. 
 Then, there exists $\mathscr{C}_{0}>0$ depending only on $d$ such that 
 the following holds for all $s>0$ and $f \in L^{2^{*}}(\mathbb{R}^{d})$:  
\begin{equation}\label{22/1/25/7:2}
\begin{split}
&
\big|
\langle 
(-\Delta +s)^{-1}  f, V\Lambda W 
\rangle 
-
\mathscr{C}_{0} \{(-\Delta +s)^{-1}f\}(0)
\big|
\lesssim 
\| f\|_{L^{2^{*}}}
,
\end{split} 
\end{equation}
where the implicit constant depends only on $d$. 
\end{lemma}
\begin{proof}[Proof of Lemma \ref{22/1/25/7:1}]
Let us begin by decomposing $\Lambda W$ as follows:  
\begin{equation}\label{18/12/03/10:17}
\Lambda W 
= 
-\frac{\{d(d-2)\}^{\frac{d}{2}}}{2d}
|x|^{-(d-2)}
+
Z,
\end{equation}
where 
\begin{equation}\label{18/12/04/09:43}
Z:=
\frac{d-2}{2}
\Big(1+\frac{|x|^{2}}{d(d-2)} \Big)^{-\frac{d}{2}}
-
\frac{|x|^{2}}{2d}
\Big\{
\Big(1+\frac{|x|^{2}}{d(d-2)} \Big)^{-\frac{d}{2}}
-
\Big(\frac{|x|^{2}}{d(d-2)} \Big)^{-\frac{d}{2}}
\Big\}
.
\end{equation} 
Let $\chi_{\le 1}$ be the indicator function of the set $\{|x|\le 1\}$, namely $\chi_{\le 1}(x)=0$ if $|x|\le 1$ and $\chi_{\le 1}(x)=0$ if $|x|>1$. 
Then, by the fundamental theorem of calculus, and $(d-\frac{3}{2})\frac{2d}{d+2}<d$ for $d=3,4$, we see that 
\begin{equation}\label{22/2/22/10:31}
\begin{split}
&\| Z \|_{L^{\frac{2d}{d+2}}}
\lesssim 
\|
(1+|x|)^{-d}
\|_{L^{\frac{2d}{d+2}}}
+
\| \int_{0}^{1} 
(\theta^{\frac{1}{2}}+|x| )^{-d} \,d \theta
\|_{L^{\frac{2d}{d+2}}}
\\[6pt]
&\lesssim  
1 + 
\| \chi_{\le 1} 
\int_{0}^{1} 
\theta^{-\frac{3}{4}} |x|^{-d+\frac{3}{2}} \,d \theta
\|_{L^{\frac{2d}{d+2}}}
+
\|(1-\chi_{\le 1}) \int_{0}^{1} 
|x|^{-d} \,d \theta
\|_{L^{\frac{2d}{d+2}}}
\lesssim 1.
\end{split} 
\end{equation}

Now, let $s>0$ and $f \in L^{2^{*}}(\mathbb{R}^{d})$. 
 By $V\Lambda W=  \Delta \Lambda W$ (see \eqref{20/9/13/9:52}), 
 the decomposition \eqref{18/12/03/10:17}, and 
 $(-\Delta) |x|^{-(d-2)}$ being the Dirac's delta measure at $0$ with intensity  $(d-2) |S^{d-1}|$ (see, e.g., Theorem 6.20 of \cite{Lieb-Loss}), we see that  
\begin{equation}\label{22/1/30/11:39}
\begin{split}
&
\langle 
(-\Delta +s)^{-1}  f, V\Lambda W 
\rangle 
=
\langle 
(-\Delta +s)^{-1}  f, \Delta \Lambda W 
\rangle 
\\[6pt]
&= 
\frac{\{d(d-2)\}^{\frac{d}{2}}}{2d}
\langle 
(-\Delta +s)^{-1} f, (-\Delta) |x|^{-(d-2)}
\rangle 
+
\langle 
(-\Delta +s)^{-1} f, \Delta Z 
\rangle
\\[6pt]
&=
\mathscr{C}_{0} 
\{ (-\Delta +s)^{-1} f \}(0)
+
\langle 
(-\Delta +s)^{-1} f, \Delta Z 
\rangle
,
\end{split} 
\end{equation}
where $\mathscr{C}_{0}>0$ is some constant depending only on $d$. 
 Consider the second term on the right-hand side of \eqref{22/1/30/11:39}.
 By H\"older's inequality, Lemma \ref{18/11/05/10:29} and \eqref{22/2/22/10:31}, we see that 
\begin{equation}\label{22/2/21/12:00}
\begin{split} 
&
\big| \langle 
(-\Delta +s)^{-1} f, \Delta Z 
\rangle 
\big| 
=
\big|
\langle 
(-\Delta +s)^{-1} f, \{s-(-\Delta+s) \} Z
\rangle 
\big|
\\[6pt]
&\le 
s \|(-\Delta +s)^{-1} f \|_{L^{2^{*}}}
\| Z \|_{L^{\frac{2d}{d+2}}}
+
\| f \|_{L^{2^{*}}}
\| Z \|_{L^{\frac{2d}{d+2}}}
\lesssim
\| f \|_{L^{2^{*}}}
.
\end{split} 
\end{equation}
Putting \eqref{22/1/30/11:39} and \eqref{22/2/21/12:00} together, 
we find that the claim \eqref{22/1/25/7:2} is true. 
\end{proof}


We will apply Lemma \ref{22/1/25/7:1} to $f=W$. 
 In that case, the principal factor is the following: 
\begin{equation}\label{22/2/21/10:51}
A_{0}:=
\left\{ \begin{array}{cc}
\sqrt{3}
\displaystyle{\int_{\mathbb{R}^{3}}
\dfrac{e^{-|x|}}{4\pi |x|^{2}}
\,dx }
& \mbox{if $d=3$},
\\[12pt]
\dfrac{\mathfrak{C}}{(2\pi)^{d}}\mathcal{F}[W^{\frac{d+2}{d-2}}](0)
=
\dfrac{\mathfrak{C}}{(2\pi)^{d}}
\displaystyle{\int_{\mathbb{R}^{d}} W^{\frac{d+2}{d-2}}(x) \,dx }
& \mbox{if $d=4$},
\end{array}\right. 
\end{equation}
where $\mathfrak{C}$ is the same constant as in \eqref{20/12/12/12:3}; Precisely, we have the following: 
\begin{lemma}\label{22/2/20/16:32}
Assume $d=3,4$. Then, the following holds for all $0< s<1$:  
\begin{equation}\label{22/2/20/16:33}
\delta(s) \{ (-\Delta +s)^{-1} W \}(0) 
= A_{0} +o_{s}(1)
,
\end{equation}
where $A_{0}$ is the same constant as in \eqref{22/2/21/10:51}. 
\end{lemma}
\begin{proof}[Proof of Lemma \ref{22/2/20/16:32}]
Divide $W$ as follows 
\begin{equation}\label{22/2/18/15:56}
W= 
\{d(d-2)\}^{\frac{d-2}{2}}|x|^{-(d-2)}  +Y 
,
\end{equation}
where 
\begin{equation}\label{22/2/20/16:48}
Y(x)
:=
\Big(1+ \frac{|x|^{2}}{d(d-2)} \Big)^{-\frac{d-2}{2}} 
-
\Big( \frac{|x|^{2}}{d(d-2)} \Big)^{-\frac{d-2}{2}}
. 
\end{equation} 
Note that 
\begin{equation}\label{22/2/21/10:59}
| Y(x)| 
\le \Big( \frac{|x|^{2}}{d(d-2)} \Big)^{-\frac{d-2}{2}} 
\lesssim |x|^{-(d-2)}
.
\end{equation}
Furthermore, by the fundamental theorem of calculus,  $\frac{4d}{2d+5}(d-1) <d$ and $\frac{4d}{2d+5}>1$ for $d=3,4$, we see that 
\begin{equation}\label{18/12/04/08:55}
\begin{split}
\|Y\|_{L^{\frac{4d}{2d+5}}}
&=
\Big\| \frac{d-2}{2}
\int_{0}^{1} 
\Big(\theta+\frac{|x|^{2}}{d(d-2)}\Big)^{-\frac{d}{2}} \,d \theta
\Big\|_{L^{\frac{4d}{2d+5}}}
\\[6pt]
&\lesssim  
\int_{0}^{1} \theta^{-\frac{1}{2}} \,d\theta
\big\| 
|x|^{-(d-1)} 
\big\|_{L^{\frac{4d}{2d+5}}(|x|\le 1)}
+
\big\| |x|^{-d}
\big\|_{L^{\frac{4d}{2d+5}}(1\le |x|)}
\lesssim 1.
\end{split} 
\end{equation}

Assume $d=3$. Then, by \eqref{22/2/18/15:56}, and the representation formula of the Green's function for $(-\Delta+s)^{-1}$ (see the remarks for Theorem 6.23 of \cite{Lieb-Loss}), we see that   
\begin{equation}\label{22/2/7/12:17}
\begin{split}
&\{(-\Delta +s )^{-1}W \}(0)
\\[6pt]
&=
\{d(d-2)\}^{\frac{d-2}{2}}
\int_{\mathbb{R}^{3}}
\frac{e^{-s_{n}^{\frac{1}{2}} |x|}}{4\pi |x|}
|x|^{-(d-2)}\,dx 
+
\int_{\mathbb{R}^{3}}
\frac{e^{-s^{\frac{1}{2}} |x|}}{4\pi |x|}
Y(x)\,dx 
.
\end{split} 
\end{equation}
Consider the first term on the right-hand side of \eqref{22/2/7/12:17}.
 By the substitution of variables and $d=3$,  we see that 
\begin{equation}\label{22/2/20/17:10}
\{d(d-2)\}^{\frac{d-2}{2}}
\int_{\mathbb{R}^{3}}
\frac{e^{-s_{n}^{\frac{1}{2}} |x|}}{4\pi |x|}
|x|^{-(d-2)}\,dx 
=
s^{-\frac{1}{2}} \sqrt{3}
\int_{\mathbb{R}^{3}}
\frac{e^{-|x|}}{4\pi |x|^{2}}
\,dx 
=
\delta(s)^{-1}A_{0}
.
\end{equation}
Consider the second term on the right-hand side of \eqref{22/2/7/12:17}.
 By \eqref{22/2/21/10:59}, $d=3$, H\"older's inequality and \eqref{18/12/04/08:55}, 
 we see that 
\begin{equation}\label{22/2/20/17:19}
\begin{split} 
&
\Big|
\int_{\mathbb{R}^{3}}
\frac{e^{-s^{\frac{1}{2}} |x|}}{4\pi |x|}
Y(x)\,dx 
\Big| 
\lesssim  
\int_{|x|\le 1}
|x|^{-2} \,dx 
+
\int_{1\le |x|}
e^{-s^{\frac{1}{2}}|x|}
|Y(x)|\,dx 
\\[6pt]
&\lesssim
1 +
s_{n}^{-\frac{d}{2}}
\int_{\mathbb{R}^{3}}
e^{-|x|}
\big| Y(s_{n}^{-\frac{1}{2}}x) \big|\,dx 
\\[6pt]
&\lesssim 
1
+ 
s^{-\frac{d}{2}}
\| e^{-|x|} \|_{L^{\frac{4d}{2d-5}}}
\|Y(s_{n}^{-\frac{1}{2}}\cdot) \|_{L^{\frac{4d}{2d+5}}}
\le
1 + s^{-\frac{2d-5}{8}}\lesssim s^{-\frac{1}{8}}.
\end{split} 
\end{equation}
Then, \eqref{22/2/20/16:33} follows from \eqref{22/2/7/12:17}, \eqref{22/2/20/17:10} and \eqref{22/2/20/17:19}.

Next, assume $d=4$. By the Fourier transform of the Green's function for $(-\Delta+s)^{-1}$ (see Theorem 6.23 of \cite{Lieb-Loss}) and $-\Delta W = W^{\frac{d+2}{d-2}}$, we see that 
\begin{equation}\label{22/2/20/17:57}
\begin{split}
&
\{(-\Delta +s )^{-1}W \}(0)
=
\mathcal{F}^{-1}[\mathcal{F}[(-\Delta+s)^{-1}W]](0)
=
\frac{1}{(2\pi)^{d}}
\int_{\mathbb{R}^{d}}
\frac{\mathcal{F}[W](\xi)}{|\xi|^{2}+s}
\,d\xi
\\[6pt]
&=
\frac{1}{(2\pi)^{d}}
\int_{\mathbb{R}^{d}}
\frac{\mathcal{F}[(-\Delta)W](\xi)}{(|\xi|^{2}+s)|\xi|^{2}}
\,d\xi
=
\frac{1}{(2\pi)^{d}}
\int_{\mathbb{R}^{d}}
\frac{\mathcal{F}[W^{\frac{d+2}{d-2}}](\xi)}{(|\xi|^{2}+s)|\xi|^{2}}
\,d\xi  
\\[6pt]
&=
\frac{1}{(2\pi)^{d}}
\int_{|\xi|\le 1}
\frac{\mathcal{F}[W^{\frac{d+2}{d-2}}](0)}{(|\xi|^{2}+s)|\xi|^{2}}
\,d\xi 
+
\frac{1}{(2\pi)^{d}}
\int_{|\xi|\le 1}
\frac{\mathcal{F}[W^{\frac{d+2}{d-2}}](\xi)-\mathcal{F}[W^{\frac{d+2}{d-2}}](0)}{(|\xi|^{2}+s)|\xi|^{2}}
\,d\xi  
\\[6pt]
&\quad +
\frac{1}{(2\pi)^{d}}\int_{1\le |\xi|}
\frac{\mathcal{F}[W^{\frac{d+2}{d-2}}](\xi)}{(|\xi|^{2}+s)|\xi|^{2}}
\,d\xi  
.
\end{split} 
\end{equation}
It follows from \eqref{20/12/12/12:3} and $W(x)\sim (1+|x|)^{-(d-2)}$ that 
\begin{equation}\label{22/2/20/18:10}
\Big|\int_{|\xi|\le 1}
\frac{\mathcal{F}[W^{\frac{d+2}{d-2}}](0)}{(|\xi|^{2}+s)|\xi|^{2}}
\,d\xi 
-
\mathfrak{C} \delta(s)^{-1} 
\mathcal{F}[W^{\frac{d+2}{d-2}}](0)
\Big| \lesssim 
\| W^{\frac{d+2}{d-2}}\|_{L^{1}}\lesssim 1
.
\end{equation}
Consider the second term on the right-hand side of \eqref{22/2/20/17:57}.
 By $W(x)\sim (1+|x|)^{-(d-2)}$, we see that 
\begin{equation}\label{22/2/21/9:15}
\begin{split} 
&
\Big|
\int_{|\xi|\le 1}
\frac{\mathcal{F}[W^{\frac{d+2}{d-2}}](\xi)-\mathcal{F}[W^{\frac{d+2}{d-2}}](0)}{(|\xi|^{2}+s)|\xi|^{2}}
\,d\xi 
\Big|
\\[6pt]
&\le    
\int_{|\xi|\le 1} 
\frac{1}{(|\xi|^{2}+s)|\xi|^{2}} 
\Big(
\int_{\mathbb{R}^{4}}
\big| e^{-ix\cdot \xi} - 1 \big|
W^{\frac{d+2}{d-2}}(x) 
\,dx  
\Big)
\,d\xi 
\\[6pt]
&\lesssim    
\int_{|\xi|\le 1} 
\frac{1}{(|\xi|^{2}+s)|\xi|^{2}} 
\Big(
\int_{\mathbb{R}^{4}}
\Big\{  
\sin^{2}\Big(\frac{x\cdot \xi}{2}\Big)
+
|\sin{(x\cdot \xi)}|
\Big\} W^{\frac{d+2}{d-2}}(x) 
\,dx  
\Big)
\,d\xi 
\\[6pt]
&\lesssim  
\int_{|\xi|\le 1} 
\frac{1}{(|\xi|^{2}+s)|\xi|^{2}} 
\Big(
\int_{\mathbb{R}^{4}}
\min\{1, |x| |\xi| \}
(1+|x|)^{-(d+2)} 
\,dx  
\Big)
\,d\xi 
\\[6pt]
&\lesssim
\int_{|\xi|\le 1} 
\frac{1}{|\xi|^{3}} 
\Big(
\int_{\mathbb{R}^{4}}
|x| (1+|x|)^{-(d+2)} 
\,dx  
\Big)
\,d\xi 
\lesssim 1
. 
\end{split} 
\end{equation}
It remains to estimate the last term on the right-hand side of \eqref{22/2/20/17:57}. By Cauchy-Schwarz inequality, Plancherel's theorem and $W(x) \sim (1+|x|)^{-(d-2)}$, we see that 
\begin{equation}\label{22/2/21/10:29}
\Big|
\int_{1\le |\xi|}
\frac{\mathcal{F}[W^{\frac{d+2}{d-2}}](\xi)}{(|\xi|^{2}+s)|\xi|^{2}}
\,d\xi 
\Big|
\le 
\Big(
\int_{1\le |\xi|}
\frac{1}{|\xi|^{8}}
\,d\xi 
\Big)^{\frac{1}{2}}
\| \mathcal{F}[W^{\frac{d+2}{d-2}}]\|_{L^{2}}
\lesssim  1
.
\end{equation} 
Then, \eqref{22/2/20/16:33} follows from \eqref{22/2/20/18:10}, 
 \eqref{22/2/21/9:15} and \eqref{22/2/21/10:29}.
\end{proof}

The following lemma follows immediately from Lemma \ref{22/1/25/7:1} and Lemma \ref{22/2/20/16:32}: 
\begin{lemma}\label{20/12/22/13:55}
Assume $d=3,4$.  Then, the following holds for all $0< s <1$:   
\begin{equation}\label{20/12/22/13:58}
\big|
\delta(s) \langle 
(-\Delta +s)^{-1} W 
, 
V\Lambda W 
\rangle 
- A_{1}
\big|   
=o_{s}(1)
\qquad 
\mbox{with}~~A_{1}:=\mathscr{C}_{0}A_{0}
,
\end{equation}
where $\mathscr{C}_{0}$ and $A_{0}$ are the same constants as in Lemma \ref{22/1/25/7:1} and \eqref{22/2/21/10:51}, respectively. 
\end{lemma}


In order to treat the term $g$ in \eqref{eq:1.1}, we need the following lemma: 
\begin{lemma}\label{18/12/19/01:00}
Assume $d=3,4$. Under Assumptions \ref{21/12/23/14:45}, the following holds for all $\omega>0$, $0< s <1$, $\mu>0$ and $\max\{1, \frac{d}{(d-2)p_{1}} \} < r <\frac{d}{2}$:  
\begin{equation}\label{18/12/20/20:57}
\begin{split} 
&
\big|
\langle (-\Delta+s )^{-1} g(\mu M_{\omega} W) , V\Lambda W \rangle
+
\langle g(\mu M_{\omega} W), \Lambda W \rangle 
\big|
\\[6pt]
&\lesssim 
\{ (\mu M_{\omega})^{p_{1}} +(\mu M_{\omega})^{p_{2}} \}
s^{\frac{d}{2r}-1}  
,
\end{split}
\end{equation}
where the implicit constant may depend on $r$ (independent of $\omega$, $s$ and $\mu$). 
 Note that $p_{1}>\frac{2}{d-2}$ implies $\frac{d}{(d-2)p_{1}}<\frac{d}{2}$.   
\end{lemma}
\begin{proof}[Proof of Lemma \ref{18/12/19/01:00}]
Let $\omega>0$, $0< s <1$ and $\mu>0$. 
 Observe from $V \Lambda W=\Delta W$ (see \eqref{20/9/13/9:52}) that 
\begin{equation}\label{18/12/20/20:25}
\begin{split}
&\langle (-\Delta+s)^{-1} g(\mu M_{\omega} W) , V\Lambda W \rangle
\\[6pt]
&=
\langle  g(\mu M_{\omega} W), (-\Delta+s)^{-1} (\Delta -s + s)\Lambda W \rangle
\\[6pt]
&=
-\langle g(\mu M_{\omega} W), \Lambda W \rangle 
+ 
s \langle  g(\mu M_{\omega} W), (-\Delta+s)^{-1} \Lambda W \rangle
.
\end{split}
\end{equation}
Let $\max\{ 1, \frac{d}{(d-2)p_{1}}\} < r <\frac{d}{2}$. 
 Then, by H\"older's inequality, \eqref{21/12/24/13:39}, 
 $W\le 1$, $W\in L^{p_{1}r}(\mathbb{R}^{d})$,  
 and \eqref{18/11/24/13:56} in Lemma \ref{18/11/05/10:29}, 
 we see that  
\begin{equation}\label{18/12/25/10:28}
\begin{split}
&s
\big| 
\langle  g(\mu M_{\omega}W), (-\Delta+s)^{-1} \Lambda W \rangle
\big|
\\[6pt]
&\lesssim  
s  
\{ (\mu M_{\omega})^{p_{1}} +(\mu M_{\omega})^{p_{2}} \}
\| W^{p_{1}} \|_{L^{r}} 
\|(-\Delta+s)^{-1}  \Lambda W \|_{L^{\frac{r}{r-1}}}
\\[6pt]
&\lesssim
s \{ (\mu M_{\omega})^{p_{1}} +(\mu M_{\omega})^{p_{2}} \}
s^{\frac{d}{2r}-2} 
\|\Lambda W \|_{L_{\rm{weak}}^{\frac{d}{d-2}}}
\lesssim
\{ (\mu M_{\omega})^{p_{1}} +(\mu M_{\omega})^{p_{2}} \}
s^{\frac{d}{2r}-1} ,
\end{split}
\end{equation}
where the implicit constant may depend on $r$.  
 Putting \eqref{18/12/20/20:25} and \eqref{18/12/25/10:28} together, 
 we obtain the desired estimate \eqref{18/12/20/20:57}. 
\end{proof}

\section{Properties of ground states}\label{18/11/05/10:01}

In this section, we give the properties of ground states which are necessary in order to prove Theorem \ref{20/8/17/11:39}.

First, note that every $H^{1}$-solution $u$ to \eqref{eq:1.1} obeys the following identity, as well as $\mathcal{N}_{\omega}(u)=0$ (see (2.1) of \cite{Berestycki-Lions}): 
\begin{equation}\label{21/12/24/14:46}
\mathcal{P}_{\omega}(u)
:=
\frac{1}{2^{*}}\|\nabla u\|_{L^{2}}^{2}
+
\frac{\omega}{2} \|u\|_{L^{2}}^{2}
-
\frac{1}{2^{*}} \|u\|_{L^{2^{*}}}^{2^{*}} 
- 
\int_{\mathbb{R}^{d}} G(|u|) =0
.
\end{equation}
Computing $\mathcal{P}_{\omega} - \dfrac{1}{2^{*}}\mathcal{N}_{\omega}$, 
 we see that  every $H^{1}$-solution $u$ to \eqref{eq:1.1} obeys 
\begin{equation}\label{18/11/24/17:41}
\omega 
\|u \|_{L^{2}}^{2}
=
d \int_{\mathbb{R}^{d}}
\Big\{ 
G(|u|) 
-\frac{1}{2^{*}} g(|u|) |u|
\Big\}
. 
\end{equation}

We give other basic properties of ground states in Section \ref{22/1/1/10:41}. Furthermore, we give key properties of ground states to proving the nondegeneracy in Section \ref{21/10/14/18:20}. 


\subsection{Basic properties of ground states}\label{22/1/1/10:41}

Under Assumption \ref{21/12/23/14:45}, we can prove several properties of grand states to \eqref{eq:1.1}. In particular, we can prove the following three lemmas (Lemma \ref{proposition:2.3}, Lemma \ref{theorem:3.1} and Lemma \ref{18/09/05/01:05}) by the same argument as in \cite{AIIKN}:

\begin{lemma}[cf. Lemma 2.3 of \cite{AIIKN}]
\label{proposition:2.3}
Assume $d\ge 3$. 
 Under Assumption \ref{21/12/23/14:45}, the following hold: 
\begin{align}
\label{21/10/14/15:22}
&
\lim_{\omega \to \infty} M_{\omega} = \infty,
\\[6pt]
\label{21/12/25/15:21}
&\lim_{\omega \to \infty} M_{\omega}^{-\frac{4}{d-2}} \omega
=0.
\end{align} 
\end{lemma}


\begin{lemma}[cf. Proposition 2.1 \cite{AIIKN}]
\label{theorem:3.1}
 Assume $d\ge 3$. 
 Under Assumption \ref{21/12/23/14:45}, 
 the following holds:
\begin{equation}\label{12/25/16:10}
\lim_{\omega \to \infty} \widetilde{\Phi}_{\omega} = W
\quad 
\mbox{strongly in $\dot{H}^{1}(\mathbb{R}^{d})$ and strongly in $C^2_{\rm loc}(\Rd)$}.
\end{equation}  
\end{lemma}

\begin{lemma}[cf. Proposition 3.1 of \cite{AIIKN}]\label{18/09/05/01:05}
Assume $d\ge 3$. 
 Under Assumption \ref{21/12/23/14:45}, there exists  $\omega_{\rm{dec}}>0$ such that for any $\omega>\omega_{\rm{dec}}$ and  $x\in \mathbb{R}^{d}$,  
\begin{equation}\label{eq:3.1}
\widetilde{\Phi}_{\omega}(x) 
\lesssim \left(1 +|x| \right)^{-(d-2)}
.
\end{equation}
\end{lemma}  

The following lemma follows from H\"older's inequality, Sobolev's one, 
 Lemma \ref{theorem:3.1} and Lemma \ref{18/09/05/01:05}:   
\begin{lemma}\label{eq:3.2}
Assume $d\ge 3$. 
 Under Assumption \ref{21/12/23/14:45}, the following holds:
\begin{equation}\label{22/1/1/11:53}
\lim_{\omega \to \infty} 
\| \widetilde{\Phi}_{\omega} -W \|_{L^{r}}
=0
\qquad 
\mbox{for all $\frac{d}{d-2} <r <\infty$}
.
\end{equation} 
\end{lemma}

Although the above lemmas give us helpful information about the ground states, they are insufficient to prove Theorem \ref{20/8/17/11:39}. We need a certain refinement of Lemma \ref{eq:3.2} (see Proposition \ref{18/12/12/13:21} in Section \ref{21/10/14/18:20}). 
 For the refinement, we prepare the following lemma: 


\begin{lemma}\label{18/11/27/11:20}
Assume $d=3,4$. 
 Under Assumption \ref{21/12/23/14:45}, 
 there exists $\omega_{1}>0$ such that for any $\omega>\omega_{1}$, there exists $\mu(\omega)=1+o_{\omega}(1)$ such that  
\begin{align}
\label{18/11/27/07:03}
&
\langle 
\big\{ 1 + (-\Delta + s(\omega)  )^{-1} V \big\}   
\zeta_{\omega}, V \Lambda W
\rangle
=0
,
\\[6pt]
\label{18/11/27/07:02}
&
\lim_{\omega \to \infty} 
\|\nabla \zeta_{\omega} \|_{L^{2}}
=0
,
\qquad 
\lim_{\omega \to \infty} 
\|\zeta_{\omega} \|_{L^{r}}
=0 
\qquad \mbox{for all $\frac{d}{d-2}<r <\infty$}
,
\end{align}
where 
\begin{equation}\label{20/12/13/15:26}
s(\omega)
:=\{ \mu(\omega) M_{\omega} \}^{-\frac{4}{d-2}} \omega
,
\qquad 
\zeta_{\omega}
:= T_{\mu(\omega) M_{\omega}}[ \Phi_{\omega}]  -W
=
T_{\mu(\omega)} 
 [ \widetilde{\Phi}_{\omega}]  -W
.
\end{equation} 
\end{lemma}
\begin{proof}[Proof of Lemma \ref{18/11/27/11:20}]
Let $\omega>0$ and $\frac{1}{2}\le \mu \le \frac{3}{2}$, and define $s_{\omega}(\mu)$ and $\zeta_{\omega}(\mu)$ as  
\begin{equation}\label{21/2/19/15:37}
s_{\omega}(\mu)
:=\{ \mu M_{\omega} \}^{-\frac{4}{d-2}} \omega
,
\qquad 
\zeta_{\omega}(\mu)
:= T_{\mu M_{\omega}} 
 \big[ \Phi_{\omega} \big]  -W
=
T_{\mu}[\widetilde{\Phi}_{\omega}]-W
.
\end{equation}
By Lemma \ref{proposition:2.3}, we may assume that
\begin{equation}\label{20/12/13/16:27}
 s_{\omega}(\mu)
\le 1
. 
\end{equation} 
Observe from $V\Lambda W = \Delta \Lambda W$ (see \eqref{20/9/13/9:52}) that 
\begin{equation}\label{18/11/28/10:12}
\begin{split}
&
\langle 
\big\{1+ (-\Delta + s_{\omega}(\mu))^{-1}V \big\}
\zeta_{\omega}(\mu),
 V \Lambda W
\rangle
\\[6pt]
&=
\langle  \zeta_{\omega}(\mu), V\Lambda W \rangle
+
\langle  (-\Delta + s_{\omega}(\mu))^{-1} V \zeta_{\omega}(\mu),\, 
s_{\omega}(\mu) \Lambda W
- 
(-\Delta + s_{\omega}(\mu)) 
\Lambda W
\rangle
\\[6pt]
&=
\langle  \zeta_{\omega}(\mu),
 s_{\omega}(\mu) V (-\Delta + s_{\omega}(\mu))^{-1} \Lambda W
\rangle
.
\end{split}
\end{equation}
Write $\zeta_{\omega}(\mu)$ as   
\begin{equation}\label{18/11/28/09:20}
\zeta_{\omega}(\mu)
= 
T_{\mu}[ \widetilde{\Phi}_{\omega} - W ]
+ 
T_{\mu}[W] -W + \frac{2}{d-2} (\mu-1)  \Lambda W
- \frac{2}{d-2} (\mu-1)  \Lambda W
.
\end{equation}
Plugging \eqref{18/11/28/09:20} into the right-hand side of \eqref{18/11/28/10:12}, and using 
 \eqref{20/9/25/11:45} and the fundamental theorem of calculus, 
 we see that 
\begin{equation}\label{20/12/13/17:11}
\begin{split}
&
\langle 
\big\{ 1+ 
(-\Delta + s_{\omega}(\mu))^{-1}V 
\big\} 
\zeta_{\omega}(\mu),\, 
 V \Lambda W
\rangle
\\[6pt]
&=
\langle
T_{\mu} [ \widetilde{\Phi}_{\omega}-W ],  
\, s_{\omega}(\mu) V  (-\Delta +s_{\omega}(\mu) )^{-1} \Lambda W
\rangle 
\\[6pt]
&\quad 
+\frac{2}{d-2}
\langle
\int_{1}^{\mu} 
\int_{1}^{\nu}
\lambda^{-2} 
T_{\lambda}
\big[  2 \Lambda W
+
\frac{2}{d-2}
 x\cdot \nabla \Lambda W
\big]
\,d \lambda d\nu 
,\, s_{\omega}(\mu) V  
(-\Delta + s_{\omega}(\mu) )^{-1} 
\Lambda W
\rangle 
\\[6pt]
&\quad -
\frac{2}{d-2} (\mu-1)
\langle
 \Lambda W
,\, s_{\omega}(\mu) V  (-\Delta + s_{\omega}(\mu) )^{-1} \Lambda W
\rangle 
.
\end{split} 
\end{equation}
By the duality,  Lemma \ref{21/1/10/15:37} and $\delta(s_{\omega}(\mu))^{-1}\ge 1$ (see \eqref{19/01/27/16:58} and \eqref{20/12/13/16:27}), 
 we see  that 
\begin{equation}\label{20/12/13/17:31}
\begin{split}
&
\| s_{\omega}(\mu) 
V  (-\Delta + s_{\omega}(\mu) )^{-1} \Lambda W
\|_{L^{\frac{2d}{d+2}}} 
= 
\sup_{\|f\|_{L^{2^{*}}}=1}
| 
\langle 
s_{\omega}(\mu) 
 V(-\Delta +  s_{\omega}(\mu) )^{-1} \Lambda W
, f 
\rangle  
|
\\
&\le 
s_{\omega}(\mu)
\sup_{\|f\|_{L^{2^{*}}}=1}
\| (-\Delta + s_{\omega}(\mu) )^{-1} \Lambda W \|_{L^{\infty}}
\|Vf\|_{L^{1}}
\\
&\lesssim 
s_{\omega}(\mu)
\sup_{\|f\|_{L^{2^{*}}}=1}
\big\{ 
\delta( s_{\omega}(\mu))^{-1} 
\|V\|_{L^{\frac{2d}{d+2}}} 
\|f\|_{L^{2^{*}}} 
\big\} 
\lesssim 
s_{\omega}(\mu) 
\delta(s_{\omega}(\mu) )^{-1} 
.
\end{split} 
\end{equation}
Consider the first term on the right-hand side of \eqref{20/12/13/17:11}. 
 By H\"older's inequality, \eqref{20/12/13/17:31}, a computation involving the scaling, $\frac{1}{2}\le \mu \le \frac{3}{2}$ and Lemma \ref{theorem:3.1}, we see that  
\begin{equation}\label{20/12/14/9:1}
\begin{split}
&
\big|
\langle
T_{\mu} [ \widetilde{\Phi}_{\omega}-W]  
,
s_{\omega}(\mu) V  (-\Delta +s_{\omega}(\mu) )^{-1} \Lambda W
\rangle 
\big|
\\[6pt]
&\lesssim 
s_{\omega}(\mu) 
\delta(s_{\omega}(\mu) )^{-1}
\|\widetilde{\Phi}_{\omega}-W  \|_{L^{2^{*}}}
=
o_{\omega}(1) 
s_{\omega}(\mu) 
\delta(s_{\omega}(\mu) )^{-1} 
.
\end{split}
\end{equation}
Consider the second term on the right-hand side of \eqref{20/12/13/17:11}.
By H\"older's inequality, \eqref{20/12/13/17:31}, a computation involving the scaling and $\frac{1}{2}\le \mu \le \frac{3}{2}$, we see that  
\begin{equation} \label{20/12/14/9:2}
\begin{split}
&
\Big|
\langle
\int_{1}^{\mu} 
\int_{1}^{\nu}
\lambda^{-2} 
T_{\lambda} \big[ 2 \Lambda W
+
\frac{2}{d-2}
 x\cdot  \nabla \Lambda W
\big]
\,d \lambda d\nu 
,\, s_{\omega}(\mu) V(-\Delta +  s_{\omega}(\mu) )^{-1} 
\Lambda W
\rangle 
\Big|
\\[6pt]
&\lesssim  
s_{\omega}(\mu) 
\delta(s_{\omega}(\mu) )^{-1} 
\int_{1}^{\mu} 
\int_{1}^{\nu}
\lambda^{-2}
\| 2 \Lambda W 
+
\frac{2}{d-2}
x\cdot \nabla \Lambda W
\|_{L^{2^{*}}}
\, d\lambda d\nu 
\\[6pt]
&\lesssim 
s_{\omega}(\mu) 
\delta(s_{\omega}(\mu) )^{-1} 
\int_{1}^{\mu} 
\int_{1}^{\nu}
\lambda^{-2}
\, d\lambda d\nu
\lesssim 
s_{\omega}(\mu) 
\delta(s_{\omega}(\mu) )^{-1}
|\mu - 1 |^{2} 
.
\end{split}
\end{equation}
Consider the last  term on the right-hand side of \eqref{20/12/13/17:11}.  
 Note that  $\mathcal{F}[V\Lambda W](0) = - \frac{d+2}{d-2}\langle W^{\frac{4}{d-2}}, \Lambda W \rangle >0$ (see \eqref{20/9/13/9:52} and \eqref{20/12/22/14:18}). Then, Lemma \ref{18/12/02/11:25} shows that there exists $\mathscr{C}>0$ such that  
\begin{equation}\label{20/12/14/9:4}
\Big| 
\frac{2}{d-2}
\langle
 \Lambda W
,\, s_{\omega}(\mu) V  (-\Delta + s_{\omega}(\mu) )^{-1} \Lambda W
\rangle 
+ 
\mathscr{C} s_{\omega}(\mu) \delta(s_{\omega}(\mu))^{-1}   
\Big| 
\lesssim 
s_{\omega}(\mu)  
.
\end{equation}  
Putting \eqref{20/12/13/17:11}, \eqref{20/12/14/9:1}, \eqref{20/12/14/9:2} and \eqref{20/12/14/9:4} together, we find that  
\begin{equation}\label{20/12/14/9:5}
\begin{split}
&
\big|
\langle 
\big\{ 1 + (-\Delta + s_{\omega}(\mu))^{-1}V \big\}  
\zeta_{\omega}(\mu),
\, V \Lambda W
\rangle
-(\mu-1)
\mathscr{C} 
s_{\omega}(\mu) \delta(s_{\omega}(\mu))^{-1}
\big|
\\[6pt]
&\lesssim   
s_{\omega}(\mu) \delta(s_{\omega}(\mu))^{-1}
\big\{ 
o_{\omega}(1)
+
|\mu - 1 |^{2} 
+
\delta(s_{\omega}(\mu))
| \mu-1 |
\big\}
.
\end{split} 
\end{equation}
Furthermore, by \eqref{20/12/14/9:5} and the intermediate value theorem, 
 we find that there exists $\omega_{1}>0$ such that if $\omega>\omega_{1}$, then there exists $\mu(\omega) >0$ such that      
\begin{align}
&
\label{20/12/14/9:6}
\mu(\omega) =1+ o_{\omega}(1)
,
\\[6pt]
\label{21/2/20/11}
&  
\langle 
\big\{ 1+ (-\Delta +  s_{\omega}(\mu(\omega))   )^{-1} V \big\} 
\zeta_{\omega}(\mu(\omega)),
\, V \Lambda W
\rangle
=
0
.
\end{align}
Then, \eqref{21/2/20/11} corresponds to \eqref{18/11/27/07:03}.

It remains to prove \eqref{18/11/27/07:02}.
  By \eqref{18/11/28/09:20}, \eqref{20/9/25/11:45}, Lemma \ref{theorem:3.1}  
 and \eqref{20/12/14/9:6}, we see that   
\begin{equation}\label{20/12/14/9:7}
\begin{split}
\|\nabla \zeta_{\omega}\|_{L^{2}}
&\le 
\| \nabla T_{\mu(\omega)}[ \widetilde{\Phi}_{\omega} -W] \|_{L^{2}}
+
\| \nabla \{ T_{\mu(\omega)}[W] -W \}\|_{L^{2}}
\\[6pt]
&\lesssim 
\| \nabla \{ \widetilde{\Phi}_{\omega} -W \} \|_{L^{2}}
+
\int_{1}^{\mu(\omega)} 
\lambda^{-1}  \|\nabla \Lambda W \|_{L^{2}} 
\,d \lambda
=
o_{\omega}(1)
.
\end{split}
\end{equation}
Similarly,  we can verify that for any $\frac{d}{d-2}<r < \infty$, 
\begin{equation}\label{20/12/14/9:8}
\begin{split}
\|\zeta_{\omega}\|_{L^{r}}
&\lesssim
\mu(\omega)^{\frac{2^{*}-r}{r}}  
\| \widetilde{\Phi}_{\omega} -W  \|_{L^{r}}
+
\int_{1}^{\mu(\omega)} 
\lambda^{-1 + \frac{2^{*}-r}{r}}  \| \Lambda W \|_{L^{r}} 
\,d \lambda
=o_{\omega}(1)
.
\end{split}
\end{equation}
Thus, we have completed the proof of the lemma. 
\end{proof}


\subsection{Key properties of ground states}
\label{21/10/14/18:20}

The aim of this subsection is to give key properties of ground states to proving Theorem \ref{20/8/17/11:39} (see Proposition \ref{18/12/12/13:21}). 

Let us begin by introducing the notation used in this subsection: 
\begin{notation}\label{21/12/29/17:3}
\begin{enumerate}
\item 
We use $\omega_{1}$ to denote the constant given by Lemma \ref{18/11/27/11:20}. 
\item 
For $\omega>\omega_{1}$, we use $\mu(\omega)$ to denote the constant given by Lemma \ref{18/11/27/11:20},

\item 
We define 
\begin{align}
\label{20/8/10/17:20}
s(\omega) 
&:=\{ \mu(\omega)M_{\omega}\}^{-\frac{4}{d-2}}\omega , 
\\[6pt]
\label{21/12/31/9:57} 
\kappa(\omega)
&:= -\langle g(\mu(\omega)M_{\omega}W), \Lambda W \rangle
,
\\[6pt]
\label{21/12/29/17:13}
t(\omega)
&:=
\{ \mu(\omega)M_{\omega}\}^{-\frac{d+2}{d-2}}
\kappa(\omega)
,
\\[6pt]  
\label{21/12/29/15:50}
Q_{\omega}
&:=
T_{\mu(\omega)}[\widetilde{\Phi}_{\omega}]
,
\qquad 
\zeta_{\omega}
:=
Q_{\omega} -W
.
\end{align}
In addition, we define 
\begin{equation}\label{19/01/13/15:16}
\begin{split}
h_{\omega}
&:=
Q_{\omega}^{\frac{d+2}{d-2}}- W^{\frac{d+2}{d-2}} 
- 
\frac{d+2}{d-2}W^{\frac{4}{d-2}} \zeta_{\omega} 
\\[6pt]
&\quad + 
t(\omega) \kappa(\omega)^{-1} 
\big\{ g(\mu(\omega) M_{\omega}Q_{\omega}) - g( \mu(\omega) M_{\omega} W) 
\big\}
. 
\end{split}
\end{equation} 
\end{enumerate} 
\end{notation}
By Lemma \ref{18/09/05/01:05} and $\mu(\omega)=1+o_{\omega}(1)$, we see that 
\begin{equation}\label{20/12/16/16:33}
\|Q_{\omega}\|_{L^{r}} \sim \|W\|_{L^{r}} \sim 1
\quad 
\mbox{for all $\frac{d}{d-2}< r \le \infty$}, 
\end{equation}
where the implicit constants may depend on $r$. 

Observe from \eqref{22/3/10/16:12} that  
$Q_{\omega}$ obeys  
\begin{equation}\label{21/12/29/15:43}
-
\Delta Q_{\omega}
+
s(\omega) 
Q_{\omega}
-
Q_{\omega}^{\frac{d+2}{d-2}}
-
t(\omega) 
\kappa(\omega)^{-1}
g(\mu(\omega)M_{\omega} Q_{\omega})
=
0
.
\end{equation}
Furthermore, it follows from  \eqref{21/12/29/15:43} and $W$ being a solution to \eqref{21/10/16/11:41} that  
\begin{equation}\label{18/11/25/15:51}
(-\Delta+s(\omega)+ V ) \zeta_{\omega}
=
-s(\omega) W 
+ 
t(\omega) \kappa(\omega)^{-1} 
g(\mu(\omega) M_{\omega}W) 
+h_{\omega}
.
\end{equation}


The following lemma gives us the limiting profile of $\kappa(\omega)$: 

\begin{lemma}\label{18/11/24/17:55}
Assume $d=3,4$. Under Assumption \ref{21/12/23/14:45} and Notation \ref{21/12/29/17:3}, the following holds:
\begin{equation}\label{21/12/31/10:31}
\lim_{\omega \to \infty}
\frac{\kappa(\omega) }{\{\mu(\omega)M_{\omega}\}^{p_{2}}} 
=
-\frac{C_{2}}{p_{2}}
\langle W^{p_{2}}, \Lambda W \rangle
.
\end{equation}
\end{lemma}
\begin{remark}\label{22/1/1/9:35}
Note that $\langle W^{p_{2}}, \Lambda W \rangle<0$ (see \eqref{20/12/22/14:18}). Furthermore, \eqref{21/12/31/10:31} implies that 
\begin{equation}\label{22/1/26/10:35}
\kappa(\omega) \sim M_{\omega}^{p_{2}}
. 
\end{equation}
\end{remark}


Now, we state key properties of ground states to proving the nondegeneracy (Theorem \ref{20/8/17/11:39}):
\begin{proposition}\label{18/12/12/13:21}
Assume $d=3,4$. Under Assumption \ref{21/12/23/14:45} and Notation \ref{21/12/29/17:3}, the following hold:  
\begin{enumerate}
\item 
\begin{equation}\label{21/10/17/8:59}
\lim_{\omega \to \infty}
s(\omega)=0,
\qquad  
\lim_{\omega \to \infty} t(\omega)=0
.
\end{equation}

\item 
\begin{equation}\label{20/12/26/17:57}
\|\zeta_{\omega}\|_{L^{r}} 
\lesssim 
s(\omega)^{\frac{d-2}{2}-\frac{d}{2r}-\varepsilon}
\quad 
\mbox{for all $\omega>\omega_{1}$, $\frac{d}{d-2}<r <\infty$ and $\varepsilon>0$}
,
\end{equation}
where the implicit constant may depend on $r$ and $\varepsilon$ (independent of $\omega$).


\item 
\begin{equation}\label{18/12/10/22:44}
|
A_{1} \beta( s(\omega))  - t(\omega) 
| 
=  
o_{\omega}(1) \beta(s(\omega))
,
\end{equation}
where $A_{1}$ is the same constant as in Lemma \ref{20/12/22/13:55}
.
\end{enumerate} 
\end{proposition}


\subsection{Proofs of Lemma \ref{18/11/24/17:55} and Proposition \ref{18/12/12/13:21}}\label{22/2/24/8:46}

\begin{proof}[Proof of Lemma \ref{18/11/24/17:55}] 
By \eqref{21/12/25/15:11}, \eqref{21/10/14/15:22} and $\mu(\omega)=1+o_{\omega}(1)$, we see that 
\begin{equation}\label{22/1/9/11:27}
\lim_{\omega\to \infty}
\Big| 
\frac{g(\mu(\omega)M_{\omega}W)}{\{\mu(\omega)M_{\omega}W\}^{p_{2}}}-\frac{C_{2}}{p_{2}}
\Big|  W^{p_{2}} \Lambda W 
=0
\qquad 
\mbox{almost everywhere in $\mathbb{R}^{d}$}
.
\end{equation}
Furthermore, by \eqref{21/12/24/13:39} and $\frac{2}{d-2}<p_{1} \le p_{2}$, we see that 
\begin{equation}\label{22/1/9/11:37}
\Big| 
\frac{g(\mu(\omega)M_{\omega}W)}{\{\mu(\omega)M_{\omega}W\}^{p_{2}}}-\frac{C_{2}}{p_{2}}
\Big|  W^{p_{2}} \Lambda W 
\lesssim
\{ W^{p_{1}}+W^{p_{2}}\}\Lambda W \in L^{1}(\mathbb{R}^{d})
.
\end{equation}
Then, the Lebesgue's dominated convergence theorem together with \eqref{22/1/9/11:27} and \eqref{22/1/9/11:37} shows that 
\begin{equation}\label{21/12/30/15:39}
\begin{split}
\Big|
\frac{\kappa(\omega)}{\{\mu(\omega)M_{\omega}\}^{p_{2}}}
+
\frac{C_{2}}{p_{2}}
\langle W^{p_{2}}, \Lambda W \rangle
\Big|
&\le 
\int_{\mathbb{R}^{d}} 
\Big| 
\frac{g(\mu(\omega)M_{\omega}W)}{\{\mu(\omega)M_{\omega}W\}^{p_{2}}}-\frac{C_{2}}{p_{2}}
\Big|  
W^{p_{2}} \Lambda W 
\\[6pt]
&\to 0 
\qquad 
\mbox{as $\omega \to \infty$}
. 
\end{split}  
\end{equation}
Thus, we have proved the lemma. 
\end{proof}


In order to prove Proposition \ref{18/12/12/13:21}, we need the following lemma:\begin{lemma}\label{22/2/24/8:36}
Assume $d=3,4$. Under Assumption \ref{21/12/23/14:45} and Notation \ref{21/12/29/17:3}, the following holds for all sufficiently large $\omega>\omega_{1}$:
\begin{equation}\label{18/12/09/22:07}
t(\omega) 
\lesssim 
s(\omega)^{\frac{d-2}{2}-\varepsilon}
\qquad 
\mbox{for all $0< \varepsilon <1$}
,
\end{equation}
where the implicit constant may depend on $\varepsilon$.
\end{lemma}
\begin{remark}\label{22/2/24/8:33}
The estimate \eqref{18/12/09/22:07} is improved to  
$t(\omega)  \sim \beta( s(\omega))$ in \eqref{18/12/10/22:44} of Proposition \ref{18/12/12/13:21}. 
\end{remark}


\begin{proof}[Proof of Lemma \ref{22/2/24/8:36}]
By \eqref{18/11/24/17:41}, \eqref{22/3/8/11:29}, \eqref{20/12/16/16:33} and 
  $t(\omega) \sim M_{\omega}^{-\frac{d+2}{d-2}+p_{2}}$ (see \eqref{21/12/29/17:13} and \eqref{22/1/26/10:35}), we see that  if $\omega$ is sufficiently large (hence Lemma \ref{proposition:2.3} implies $\mu(\omega)M_{\omega}\gg 1$), then 
\begin{equation}\label{20/12/16/16:48}
\begin{split} 
\|Q_{\omega}\|_{L^{2}}^{2}
&=
s(\omega)^{-1}
\{\mu(\omega)M_{\omega}\}^{-2^{*}} 
d \int_{\mathbb{R}^{d}}
\Big\{ G(\mu(\omega)M_{\omega} Q_{\omega}) 
- 
\frac{1}{2^{*}} g(\mu(\omega)M_{\omega} Q_{\omega}) 
\mu(\omega)M_{\omega} Q_{\omega}
\Big\} 
\\[6pt]
&\ge  
 s(\omega)^{-1}
\{\mu(\omega)M_{\omega}\}^{-2^{*}} 
dC_{3}
\{\mu(\omega)M_{\omega}\}^{p_{2}+1}
\|Q_{\omega} \|_{L^{p_{2}+1}}^{p_{2}+1}
\\[6pt]
&\gtrsim
s(\omega)^{-1} M_{\omega}^{-\frac{d+2}{d-2}}
M_{\omega}^{p_{2}}
\sim 
s(\omega)^{-1} t(\omega)
.
\end{split}
\end{equation}
Hence, the claim \eqref{18/12/09/22:07} follows from \eqref{20/12/16/16:48} and the following estimate:
\begin{equation}\label{18/11/24/17:57}
\|Q_{\omega}  \|_{L^{2}}^{2}
\lesssim 
s(\omega)^{\frac{d-2}{2}-1-\varepsilon}
\quad 
\mbox{for all $0< \varepsilon <1$}
.
\end{equation}

We shall prove \eqref{18/11/24/17:57}. Let $0< \varepsilon <1$. We allow the implicit constants to depend on $\varepsilon$. 
 Observe from \eqref{21/12/29/15:43} that 
\begin{equation}\label{20/12/18/9:47}
\begin{split}
\|Q_{\omega} \|_{L^{2}}^{2}
&=
\langle
(-\Delta +s(\omega))^{-1} 
(-\Delta +s(\omega))Q_{\omega} ,Q_{\omega} 
\rangle 
\\[6pt]
&=\langle 
(-\Delta +s(\omega))^{-1}
Q_{\omega}^{\frac{d+2}{d-2}}, 
Q_{\omega} 
\rangle 
\\[6pt]
&\quad +
t(\omega)
\langle 
(-\Delta + s(\omega))^{-1} 
\kappa(\omega)^{-1} 
g(\mu(\omega) M_{\omega} Q_{\omega}), Q_{\omega} 
\rangle 
.
\end{split}
\end{equation}

Consider the first term on the right-hand side of \eqref{20/12/18/9:47}. 
 By Cauchy-Schwarz inequality, Lemma \ref{18/11/05/10:29}, and \eqref{20/12/16/16:33}, we see that  
\begin{equation}\label{18/12/26/11:39}
\begin{split}
&\big|
\langle 
(-\Delta +s(\omega))^{-1}
Q_{\omega}^{\frac{d+2}{d-2}}, Q_{\omega} 
\rangle 
\big| 
\le 
\| (-\Delta+s(\omega))^{-1} 
Q_{\omega}^{\frac{d+2}{d-2}}
\|_{L^{2}}
\|  Q_{\omega} \|_{L^{2}}
\\[6pt]
&\lesssim 
s(\omega)^{\frac{d}{4}-1-\frac{\varepsilon}{2}} 
\|
Q_{\omega}^{\frac{d+2}{d-2}} 
\|_{L^{\frac{d}{d-\varepsilon}}}
\|Q_{\omega}   \|_{L^{2}}
\lesssim 
s(\omega)^{\frac{1}{2}(\frac{d-2}{2}-1-\varepsilon)} 
\|Q_{\omega}   \|_{L^{2}}
.
\end{split}  
\end{equation} 
Move on to the second term on the right-hand side of \eqref{20/12/18/9:47}. 
Fix $0< \varepsilon_{2}< \frac{(d-2)p_{1}-2}{d}$, so that $\frac{p_{1}+1}{1+\varepsilon_{2}}>\frac{d}{d-2}$. 
 Then, by \eqref{22/1/26/10:35}, H\"older's inequality,  
Lemma \ref{18/11/05/10:29},  \eqref{21/12/24/13:39},  \eqref{20/12/16/16:33}
 and \eqref{20/12/16/16:48}, we see that 
\begin{equation}\label{18/11/25/09:01}
\begin{split}
&
t(\omega)
\big| 
\langle 
(-\Delta + s(\omega))^{-1} 
\kappa(\omega)^{-1} 
g(\mu(\omega) M_{\omega} Q_{\omega}), Q_{\omega} 
\rangle 
\big| 
\\[6pt]
&\lesssim  
t(\omega)
M_{\omega}^{-p_{2}}
\|
(-\Delta + s(\omega))^{-1} 
g(\mu(\omega) M_{\omega} Q_{\omega})
\|_{L^{\frac{p_{1}+1}{p_{1}-\varepsilon_{2}}}}  
\| Q_{\omega} \|_{L^{\frac{p_{1}+1}{1+\varepsilon_{2}}}}
\\[6pt]
&\lesssim 
t(\omega)
M_{\omega}^{-p_{2}}
s(\omega)^{\frac{d\varepsilon_{2}}{2(p_{1}+1)}-1}  
\|
g (\mu(\omega)M_{\omega} Q_{\omega})
\|_{L^{\frac{p_{1}+1}{p_{1}}}}
\\[6pt]
&\lesssim 
t(\omega)
M_{\omega}^{-p_{2}}
s(\omega)^{\frac{d\varepsilon_{2}}{2(p_{1}+1)}-1}
M_{\omega}^{p_{2}}
\lesssim 
s(\omega)^{\frac{d\varepsilon_{2}}{2(p_{1}+1)}}
\|Q_{\omega} \|_{L^{2}}^{2}
=
o_{\omega}(1) \|Q_{\omega} \|_{L^{2}}^{2}
.
\end{split}  
\end{equation}
Putting  \eqref{20/12/18/9:47}, \eqref{18/12/26/11:39} and \eqref{18/11/25/09:01} together, we find that the claim \eqref{18/11/24/17:57} is true.   Thus, we have proved the lemma. 
\end{proof}


Now, we are in a position to prove  Proposition \ref{18/12/12/13:21}: 
\begin{proof}[Proof of Proposition \ref{18/12/12/13:21}]
We use the symbols in Notation \ref{21/12/29/17:3}. 


The claim \eqref{21/10/17/8:59} follows immediately from Lemma \ref{proposition:2.3} and \eqref{22/1/26/10:35}.


We shall prove \eqref{20/12/26/17:57}. Let $\omega>\omega_{1}$, $\frac{d}{d-2}<r <\infty$ and $0< \varepsilon <1$. We allow the implicit constants to depend on $r$ and $\varepsilon$. 
 By \eqref{18/11/11/15:50} in Proposition \ref{18/11/17/07:17} together with \eqref{18/11/27/07:03} in Lemma \ref{18/11/27/11:20}, and \eqref{18/11/25/15:51}, we see that   \begin{equation}\label{18/11/25/16:45}
\begin{split}
\|\zeta_{\omega} \|_{L^{r}}
&\lesssim 
\| \{ 1+(-\Delta + s(\omega) )^{-1}V \}
\zeta_{\omega} 
\|_{L^{r}}
\\[6pt]
&=
\| (-\Delta + s(\omega) )^{-1} (-\Delta+s(\omega) +V ) 
\zeta_{\omega} 
\|_{L^{r}}
\\[6pt]
&\lesssim  
s(\omega)
\|(-\Delta+s(\omega))^{-1} W \|_{L^{r}} 
+
t(\omega) \kappa(\omega)^{-1}
\|(-\Delta+s(\omega))^{-1} g(\mu(\omega) M_{\omega}W) 
\|_{L^{r}} 
\\[6pt]
&\quad +
\| (-\Delta+s(\omega))^{-1} h_{\omega} 
\|_{L^{r}}
.
\end{split} 
\end{equation}
Consider the first term on the right-hand side of \eqref{18/11/25/16:45}. 
 By Lemma \ref{18/11/05/10:29}, we see that  
\begin{equation}\label{18/11/25/15:27}
\begin{split}
s(\omega) \|(-\Delta+s(\omega))^{-1} W \|_{L^{r}} 
&\lesssim 
s(\omega)^{\frac{d-2}{2}-\frac{d}{2r}}
\|W\|_{L_{\rm weak}^{\frac{d}{d-2}}}
\lesssim 
s(\omega)^{\frac{d-2}{2}-\frac{d}{2r}}
.
\end{split}
\end{equation}
Move on to the second term on the right-hand side of \eqref{18/11/25/16:45}. 
 By Lemma \ref{22/2/24/8:36}, \eqref{22/1/26/10:35}, 
Lemma \ref{18/11/05/10:29} and \eqref{21/12/24/13:39}, we see that 
\begin{equation}\label{18/12/07/10:37}
\begin{split}
&t(\omega)\kappa(\omega)^{-1}
\| (-\Delta+s(\omega))^{-1} g(\mu(\omega) M_{\omega}W)
\|_{L^{r}} 
\\[6pt]
&\lesssim 
s(\omega)^{\frac{d-2}{2}-\varepsilon}
M_{\omega}^{-p_{2}}
s(\omega)^{-\frac{d}{2r}} 
\| g(\mu(\omega) M_{\omega}W) \|_{L^{\frac{d}{2}}}
\\[6pt]
&\lesssim 
s(\omega)^{\frac{d-2}{2}-\frac{d}{2r}-\varepsilon}
M_{\omega}^{-p_{2}}
\sum_{j=1,2} 
M_{\omega}^{p_{j}}\|W \|_{L^{\frac{dp_{j}}{2}}}^{p_{j}}
\lesssim
s(\omega)^{\frac{d-2}{2}-\frac{d}{2r}-\varepsilon}
.
\end{split}
\end{equation}
Consider the last term on the right-hand side of \eqref{18/11/25/16:45}. 
 By  Lemma \ref{18/11/23/17:17}, Lemma \ref{18/11/05/10:29}, a computation involving the fundamental theorem of calculus, H\"older's inequality, 
 \eqref{22/1/26/10:35}, Lemma \ref{22/2/24/8:36} and \eqref{21/12/24/13:39}, we see that 
\begin{equation}\label{20/12/22/16}
\begin{split}
&\| (-\Delta+s(\omega) )^{-1} h_{\omega} \|_{L^{r}}
\\[6pt]
&\lesssim 
\|Q_{\omega}^{\frac{d+2}{d-2}}-W^{\frac{d+2}{d-2}} - \frac{d+2}{d-2}W^{\frac{4}{d-2}}\zeta_{\omega} \|_{L^{\frac{dr}{d+2r}}}
\\[6pt]
&\quad +
t(\omega)\kappa(\omega)^{-1}
s(\omega)^{-\frac{d}{2r}} 
\|g( \mu(\omega)M_{\omega}Q_{\omega}) -g( \mu(\omega)M_{\omega}W) \|_{L^{\frac{d}{2}}}
\\[6pt]
&\lesssim
\| 
|W|^{\frac{6-d}{d-2}} |\zeta_{\omega} |^{2} + |\zeta_{\omega} |^{\frac{d+2}{d-2}}
\|_{L^{\frac{dr}{d+2r}}}
\\[6pt]
&\quad +
t(\omega) \kappa(\omega)^{-1}
s(\omega)^{-\frac{d}{2r}} 
\| \int_{0}^{1} g'( \mu(\omega)M_{\omega}\{ W+\theta \zeta_{\omega}\})
d\theta 
\mu(\omega)M_{\omega} \zeta_{\omega}
\|_{L^{\frac{d}{2}}} 
\\[6pt]
&\lesssim  
\| \zeta_{\omega} \|_{L^{2^{*}}}
\| \zeta_{\omega}\|_{L^{r}}
+ 
\| \zeta_{\omega} \|_{L^{2^{*}}}^{\frac{4}{d-2}}
\| \zeta_{\omega}\|_{L^{r}}
\\[6pt]
&\quad +
s(\omega)^{\frac{d-2}{2}-\frac{d}{2r}-\varepsilon}
M_{\omega}^{-p_{2}}
\sum_{j=1,2}M_{\omega}^{p_{j}}\big\{
\| W^{p_{j}-1} \zeta_{\omega} \|_{L^{\frac{d}{2}}}
+
\| \zeta_{\omega}  \|_{L^{\frac{dp_{j}}{2}}}^{p_{j}}
\big\}
.
\end{split}
\end{equation}
Here, for $j=1,2$, define $r(p_{j})$ as 
\begin{equation}\label{22/1/1/17:31}
r(p_{j}):=
\left\{ \begin{array}{ccc}
\frac{d}{2-(d-2)(p_{j}-1)} & \mbox{if} & \frac{2}{d-2}< p_{j}<\frac{d}{d-2},
\\[6pt]
\infty &\mbox{if}& \frac{d}{d-2}\le p_{j} < \frac{d+2}{d-2}
.
\end{array} \right.
\end{equation} 
Note that $\frac{d}{d-2}<r(p_{j})$. Fix $\frac{d}{d-2}<r_{j}<r(p_{j})$. 
Then, by H\"older's inequality, we see that if $d=3,4$ (hence $\frac{d}{2}\le \frac{d}{d-2}$), then 
\begin{equation}\label{20/12/20/16:41}
\| W^{p_{j}-1} \zeta_{\omega} \|_{L^{\frac{d}{2}}}
\le 
\| W \|_{L^{\frac{dr_{j}(p_{j}-1)}{2r_{j}-d}}}^{p_{j}-1} 
\|\zeta_{\omega}\|_{L^{r_{j}}} 
\lesssim
\|\zeta_{\omega}\|_{L^{r_{j}}}
. 
\end{equation}
Plugging \eqref{20/12/20/16:41} into \eqref{20/12/22/16}, 
 and using \eqref{18/11/27/07:02} in Lemma \ref{18/11/27/11:20}, we see that  
\begin{equation}\label{18/12/06/15:17}
\begin{split}
&
\| (-\Delta+s(\omega))^{-1} h_{\omega} 
\|_{L^{r}}
= 
o_{\omega}(1) 
\{ \|\zeta_{\omega} \|_{L^{r}}
+
s(\omega)^{\frac{d-2}{2}-\frac{d}{2r}-\varepsilon}
\}
.
\end{split}
\end{equation} 
Furthermore, putting \eqref{18/11/25/16:45}, \eqref{18/11/25/15:27}, \eqref{18/12/07/10:37} and \eqref{18/12/06/15:17} together, we find that 
 \eqref{20/12/26/17:57} holds, namely, it holds that  
\begin{equation}\label{21/1/2/14:42}
\|\zeta_{\omega}\|_{L^{r}} 
\lesssim 
s(\omega)^{\frac{d-2}{2}-\frac{d}{2r}-\varepsilon}
\qquad 
\mbox{for all $\frac{d}{d-2}<r<\infty$ and $\varepsilon>0$}
,
\end{equation}
where the implicit constant may depend on $r$ and $\varepsilon$.


We shall prove the last claim \eqref{18/12/10/22:44}. 
 To this end, we introduce the following symbols:  
\begin{align}
\label{18/12/23/15:35}
\mathscr{X}(\omega)&:= 
\delta(s(\omega)) 
\langle 
(-\Delta + s(\omega))^{-1} W, V \Lambda W
\rangle
,
\\[6pt]
\label{18/12/23/15:51}
\mathscr{K}(\omega)
&:= 
\kappa(\omega)^{-1} \langle 
(-\Delta + s(\omega) )^{-1} g(\mu(\omega)M_{\omega}W), V \Lambda W
\rangle
,
\\[6pt]
\label{20/12/23/9:50}
\mathscr{H}(\omega)
&:=
\langle 
(-\Delta + s(\omega) )^{-1} h_{\omega}, V \Lambda W
\rangle
.
\end{align}
Observe from $\beta(s)=\delta(s)^{-1}s$ (see \eqref{19/01/15/08:25}), \eqref{18/11/25/15:51}, and \eqref{18/11/27/07:03} in Lemma \ref{18/11/27/11:20} that   
\begin{equation}\label{18/12/10/10:07}
\beta(s(\omega)) 
\mathscr{X}(\omega)
-
t(\omega)
\mathscr{K}(\omega)
=
\mathscr{H}(\omega)
. 
\end{equation}
Fix $\max\{1, \frac{d}{(d-2)p_{1}}\} < r_{0} <\frac{d}{2}$ and $0<\varepsilon <\frac{d}{2r_{0}}-1$. Then, by the definition of $\kappa(\omega)$ (see \eqref{21/12/31/9:57}), 
 Lemma \ref{18/12/19/01:00} and \eqref{22/1/26/10:35}, we see that  
\begin{equation}\label{18/12/12/15:52}
\begin{split}
&|  \mathscr{K}(\omega) - 1 | 
\\[6pt]
&=
\kappa(\omega)^{-1}
|\langle 
(-\Delta + s(\omega) )^{-1} g(\mu(\omega)M_{\omega}W), V \Lambda W
\rangle 
+ \langle g(\mu(\omega)M_{\omega}W), \Lambda W \rangle
| 
\\[6pt]
&\lesssim
\kappa(\omega)^{-1} M_{\omega}^{p_{2}}s(\omega)^{\frac{d}{2r_{0}}-1}
\lesssim s(\omega)^{\frac{d}{2r_{0}}-1}
.
\end{split} 
\end{equation}
Furthermore, by  \eqref{18/12/10/10:07}, Lemma \ref{20/12/22/13:55}, \eqref{18/12/12/15:52} and Lemma \ref{22/2/24/8:36}, we see that  
\begin{equation}\label{18/12/23/16:11}
\begin{split}
&
\big|
A_{1} \beta(s(\omega))- t(\omega)
\big|
= 
\big| 
\beta(s(\omega))  \{ \mathscr{X}(\omega) - A_{1} \}
-
t(\omega)\{  \mathscr{K}(\omega)-1  \} 
-
\mathscr{H}(\omega)
\big|
\\[6pt]
&\lesssim 
o_{\omega}(1) \beta(s(\omega))
+
s(\omega)^{\frac{d-2}{2}-\varepsilon+\frac{d}{2r_{0}}-1}
+
| \mathscr{H}(\omega) |
. 
\end{split}
\end{equation}
Since $\beta(s)=\delta(s)^{-1}s$ (see \eqref{19/01/15/08:25}) and $\frac{d}{2r_{0}}-1 > \varepsilon $,
 the second term on the right-hand side of \eqref{18/12/23/16:11} is
 acceptable to \eqref{18/12/10/22:44}. 
 Hence, it suffices to show that there exists $\varepsilon_{0}>0$ such that 
\begin{equation}\label{18/12/23/17:01}
| \mathscr{H}(\omega) |
\lesssim  
s(\omega)^{\frac{d-2}{2}+\varepsilon_{0}}
.
\end{equation}
We shall prove \eqref{18/12/23/17:01}. 
 For $j=1,2$, let $\frac{d}{d-2}< r_{j} < r(p_{j})$, where $r(p_{j})$ is the same constant as in \eqref{22/1/1/17:31}. 
 Furthermore, let $\frac{d}{d-2}<r<\infty$ and $\varepsilon>0$.    
 Then, by H\"older's inequality, $|V\Lambda W|\lesssim (1+|x|)^{-(d+2)}$, 
 \eqref{20/12/22/16}, \eqref{20/12/20/16:41} and \eqref{21/1/2/14:42},
 we see that 
\begin{equation}\label{22/1/2/14:23}
\begin{split}
| \mathscr{H}(\omega) |
&\le 
\| (-\Delta+s(\omega) )^{-1} h_{\omega} \|_{L^{r}}
\|V\Lambda W \|_{L^{\frac{r}{r-1}}}
\\[6pt]
&\lesssim  
\| \zeta_{\omega} \|_{L^{2^{*}}}
\| \zeta_{\omega}\|_{L^{r}}
+ 
\| \zeta_{\omega} \|_{L^{2^{*}}}^{\frac{4}{d-2}}
\| \zeta_{\omega}\|_{L^{r}}
\\[6pt]
&\quad +
s(\omega)^{\frac{d-2}{2}-\frac{d}{2r}-\varepsilon}
\sum_{j=1,2} 
\big\{
\|  \zeta_{\omega} \|_{L^{r_{j}}}
+
\| \zeta_{\omega}  \|_{L^{\frac{dp_{j}}{2}}}^{p_{j}}
\big\}
\\[6pt]
&\lesssim 
s(\omega)^{\frac{d-2}{2}+\frac{d-2}{4}-\frac{d}{2r}-2\varepsilon}
+
s(\omega)^{\frac{d-2}{2}+1-\frac{d}{2r}-\frac{d+2}{d-2}\varepsilon}
\\[6pt]
&\quad +
s(\omega)^{\frac{d-2}{2}-\frac{d}{2r}-\varepsilon}
\max_{j=1,2}\bigm\{
s(\omega)^{\frac{d-2}{2}-\frac{d}{2r_{j}}-\varepsilon}
+
s(\omega)^{\frac{(d-2)p_{j}}{2}-1-\varepsilon p_{j}}
\bigm\},
\end{split}
\end{equation}
where the implicit constants may depend on $r$ and $\varepsilon$. 
 Note that for $j=1,2$, $\frac{d}{2r_{j}}<\frac{d-2}{2}$ and $\frac{(d-2)p_{j}}{2}>1$. Hence, \eqref{22/1/2/14:23} together with an appropriate choice of $r$ and $\varepsilon$ shows that \eqref{18/12/23/17:01} holds.  
 Thus, we have completed the proof of Proposition \ref{18/12/12/13:21}. 
\end{proof}


\section{Proof of Theorem \ref{20/8/17/11:39}}
\label{18/09/09/17:30}

In this section,  we prove Theorem \ref{20/8/17/11:39}. 
 We give a proof of the claim {\it (i)} in Section \ref{22/2/23/13:19} and the claim {\it (ii)} in Section \ref{21/10/9/13:18}.  


\subsection{Nondegeneracy of ground states}\label{22/2/23/13:19}
In this section, we prove the claim {\it (i)} of Theorem \ref{20/8/17/11:39}. We will use the symbols in Notation \ref{21/12/29/17:3}.
 
\begin{proof}[Proof of Theorem \ref{20/8/17/11:39}]
Suppose for contradiction that the claim {\it (i)} of Theorem \ref{20/8/17/11:39} is false. Then, we can take a sequence $\{ \omega_{n} \}$ in $(\omega_{1},\infty)$ with the following properties:  $\lim_{n\to \infty}\omega_{n}=\infty$; and for any $n\ge 1$, there exists  a nontrivial real-valued radial function $z_{n} \in H_{\rm rad}^{1}(\mathbb{R}^{d})$ such that  
\begin{equation}\label{18/12/08/12:25}
\big\{ -\Delta  +  \omega_{n}
-
\frac{d+2}{d-2}
\Phi_{\omega_{n}}^{\frac{4}{d-2}} 		
-
g'(\Phi_{\omega_{n}}) 
\big\} 
z _{n}=0 
\qquad
\mbox{in the weak sense}
.   
\end{equation}

We see that $z_{n} \in C^{2}(\mathbb{R}^{d})$ (see, e.g., Theorem 11.7 of \cite{Lieb-Loss}), which together with $z_{n}$ being nontrivial and the Strauss' radial lemma (see Lemma 1 of \cite{Strauss}) shows that $0< \|z_{n}\|_{L^{\infty}}<\infty$. Thus, we may assume that  
\begin{equation}\label{18/12/08/12:24}
\|z_{n}\|_{L^{\infty}}= 1
\qquad 
\mbox{for all $n\ge 1$}
.
\end{equation}
Furthermore, by Lemma \ref{18/09/05/01:05}, Lemma \ref{18/11/24/17:55} and Proposition \ref{18/12/12/13:21}, we may assume that for any $n\ge 1$, there exists $\mu_{n}=1+o_{n}(1)$  such that,  
defining  $s_{n}$, $\kappa_{n}$, $t_{n}$,  $Q_{n}$ and $\zeta_{n}$ as 
\begin{align}
\label{18/12/08/12:02}
s_{n}&:= \{ \mu_{n} M_{\omega_{n}}\}^{-\frac{4}{d-2}}\omega_{n},
\\[6pt]
\label{22/1/3/11:23}
\kappa_{n}
&:=-\langle g(\mu_{n}M_{\omega_{n}}W), \Lambda W \rangle,
\qquad 
t_{n}:=\{ \mu_{n} M_{\omega_{n}} \}^{-\frac{d+2}{d-2}}
\kappa_{n}
,
\\[6pt]
\label{20/12/25/10:55}
Q_{n}
&:=T_{ \mu_{n}}[\widetilde{\Phi}_{\omega_{n}}]
,
\qquad 
\zeta_{n}
:= 
Q_{n} - W
,
\end{align}
we have the following:
\begin{enumerate}
\item 
\begin{equation}\label{20/12/28/1}
\lim_{n\to \infty} s_{n}=0,
\qquad 
\lim_{n\to \infty} t_{n}
=0
.
\end{equation}


\item 
\begin{equation}\label{19/01/01/12:18}
|Q_{n}(x)| \lesssim (1+|x|)^{-(d-2)},
\qquad 
\|Q_{n}\|_{L^{r}} \sim 1
\qquad 
\mbox{for all $\frac{d}{d-2}< r \le \infty$}, 
\end{equation}
where the implicit constant in the second claim  may depend on $r$.


\item 
\begin{equation}\label{20/12/28/10:26}
\|\zeta_{n}\|_{L^{r}} 
\lesssim 
s_{n}^{\frac{d-2}{2}-\frac{d}{2r}-\varepsilon} 
\quad 
\mbox{for all $\frac{d}{d-2} < r <\infty$ and $\varepsilon>0$}
,
\end{equation}
where the implicit constant may depend on $r$ and $\varepsilon$.


\item 
There exists $\varepsilon_{0}>0$ such that 
\begin{equation}\label{21/4/9/9:45}
| A_{1} \beta(s_{n})  - t_{n} | 
\lesssim  
o_{n}(1) \beta(s_{n})
,
\end{equation} 
where $A_{1}$ is the same constant as in Lemma \ref{20/12/22/13:55}; We may write \eqref{21/4/9/9:45} as 
\begin{equation}\label{21/4/28/11:48}
\big| A_{1} \delta(s_{n})^{-1} s_{n} -t_{n} \big|
=
o_{n}(1)t_{n}
.
\end{equation}

\item 
\begin{equation}\label{22/1/3/14:5}
\lim_{n\to \infty}
\frac{\kappa_{n}}{\{\mu_{n}M_{\omega_{n}}\}^{p_{2}}}
=
-\frac{C_{2}}{p_{2}}
\langle W^{p_{2}}, \Lambda W \rangle 
,
\qquad 
\kappa_{n} \sim M_{\omega_{n}}^{p_{2}}.
\end{equation}
\end{enumerate}

We define $\widetilde{z}_{n}$ by  
\begin{equation}\label{18/12/08/11:10}
\widetilde{z}_{n}(x):=z_{n} ( \{\mu_{n} M_{\omega_{n}} \}^{-\frac{2}{d-2}}x ).
\end{equation}
This transformation preserves the norm of $L^{\infty}(\mathbb{R}^{d})$.  
Hence, \eqref{18/12/08/12:24} shows that   
\begin{equation}\label{18/09/12/07:19}
\| \widetilde{z}_{n}  \|_{L^{\infty}} 
=
1
\qquad 
\mbox{for all $n\ge 1$}
.
\end{equation}
Observe from \eqref{21/12/29/15:43} and \eqref{18/12/08/12:25} that 
\begin{align}
\label{21/5/12/5:26}
-\Delta  Q_{n} 
+
s_{n} Q_{n}
-
Q_{n}^{\frac{d+2}{d-2}}
-
t_{n} \kappa_{n}^{-1} 
g(\mu_{n} M_{\omega_{n}} Q_{n})  
&=
0,
\\[6pt]
\label{eq:4.30}
\big\{ -\Delta + s_{n}   
- 
\frac{d+2}{d-2}
Q_{n}^{\frac{4}{d-2}} 
-
\{\mu_{n}M_{\omega_{n}}\}^{-\frac{4}{d-2}}
g'(\mu_{n}M_{\omega_{n}}Q_{n}) 
\big\}\,  
\widetilde{z}_{n} 
&=0.
\end{align}

We give an outline of how to derive a contradiction:


\noindent 
{\bf Overview of the proof.}~We break the proof into small claims:


\noindent 
{\bf Claim 1.}~There exist a subsequence of $\{\widetilde{z}_{n}\}$ (which we continue to denote by $\{\widetilde{z}_{n}\}$) and a non-zero constant $c_{\infty} \neq 0$ such that 
\begin{equation}\label{22/2/23/6:1} 
\lim_{n\to \infty}\widetilde{z}_{n}
=
c_{\infty} \Lambda W 
\quad \mbox{weakly in $\dot{H}^{1}(\mathbb{R}^{d})$ and strongly in $C_{\rm{loc}}^{1}(\mathbb{R}^{d})$}
.
\end{equation}

 
\noindent 
{\bf Claim 2.}~The following identity holds for all $n\ge 1$: 
\begin{equation}\label{proof-eq-0}
\frac{s_{n}}{t_{n}}
\langle Q_{n} , \widetilde{z}_{n} \rangle  
= 
\frac{\{\mu_{n}M_{\omega_{n}}\}^{p_{2}}}{\kappa_{n}} 
\rho_{n} 
,
\end{equation}
where 
\begin{equation}\label{22/1/5/15:41}
\rho_{n}
:=
\{\mu_{n}M_{\omega_{n}}\}^{-p_{2}}\langle 
\frac{d+2}{4} g(\mu_{n}M_{\omega_{n}}Q_{n})
-
\frac{d-2}{4}  
g'(\mu_{n}M_{\omega_{n}}Q_{n}) \mu_{n}M_{\omega_{n}}Q_{n}
, \, \widetilde{z}_{n} \rangle
.
\end{equation}

\noindent 
{\bf Claim 3.}~The left-hand side of \eqref{proof-eq-0} obeys 
\begin{equation}\label{21/4/8/11:31}
c_{\infty}
\lim_{n \to \infty} \frac{s_{n}}{t_{n}}
\langle Q_{n},  \widetilde{z}_{n} \rangle 
=
-
\frac{(d+2)-(d-2)p_{2}}{4} c_{\infty}^{2}  
,
\end{equation}
where $c_{\infty} \neq 0$ is the same constant as in \eqref{22/2/23/6:1}. We remark that the factor $c_{\infty}$ on the left-hand side of \eqref{21/4/8/11:31} is necessary to fix the sign of the right-hand side.


Before stating the other claims, we introduce the function $Q_{n}^{\perp}$ as  
\begin{equation}\label{18/09/12/07:25}
Q_{n}^{\perp} 
:=
Q_{n} - \tau_{n} V \Lambda W  
\qquad 
\mbox{with}
\quad  
\tau_{n}:= 
\frac{\langle  (-\Delta + s_{n})^{-1} Q_{n},V \Lambda W \rangle}
{\langle  (-\Delta +s_{n})^{-1} V \Lambda W,V\Lambda W \rangle}
.
\end{equation}
Note that 
\begin{equation}\label{18/09/12/07:26}
\langle (-\Delta + s_{n})^{-1} Q_{n}^{\perp}, V \Lambda W \rangle=0
.
\end{equation}
Furthermore, observe from $\dfrac{s_{n}}{t_{n}}=(1+o_{n}(1))\dfrac{\delta(s_{n})}{A_{1}}$ (see \eqref{21/4/28/11:48}) that 
\begin{equation}\label{21/10/16/17:58}
\frac{s_{n}}{t_{n}}
\langle Q_{n},\widetilde{z}_{n} \rangle 
=
\frac{1}{A_{1}}\{ 1 +o_{n}(1) \} 
\big\{
\delta(s_{n})
\langle \tau_{n} V\Lambda W, \widetilde{z}_{n} \rangle
+
\delta(s_{n})
\langle Q_{n}^{\perp}, \widetilde{z}_{n} \rangle 
\big\}
.
\end{equation}


The rest of the claims are the following: 


\noindent 
{\bf Claim 4.}
\begin{equation}\label{21/4/5/10:16}
\lim_{n\to \infty}
\langle  (-\Delta +s_{n})^{-1} V \Lambda W,V\Lambda W \rangle
=
-\langle \Lambda W, V\Lambda W \rangle
>0
.
\end{equation}

\noindent 
{\bf Claim 5.}~Assume $d=3$. Then, for each $0< \varepsilon <1$, the following holds:
\begin{equation}\label{21/4/5/10:25}
\lim_{n\to \infty}
\delta(s_{n}) \langle  (-\Delta +s_{n})^{-1} Q_{n},V \Lambda W \rangle
\ge 
\frac{1-\varepsilon}{2} A_{1}  
.
\end{equation}

\noindent 
{\bf Claim 6.}~Assume $d=4$. Then, the following holds:

\begin{equation}\label{22/2/22/16:31}
\lim_{n\to \infty}
\delta(s_{n}) \langle  (-\Delta +\widetilde{s}_{n})^{-1} Q_{n},V \Lambda W \rangle
=
A_{1} 
.
\end{equation}

\noindent 
{\bf Claim 7.}~ 
\begin{equation}\label{21/4/9/10:48}
\lim_{n \to \infty}
\delta(s_{n})
\langle Q_{n}^{\perp}, \widetilde{z}_{n} \rangle
=
0. 
\end{equation}


Now, accepting the above claims for the time being, we derive a contradiction: 


\noindent 
{\bf Completion of the proof.}~Assume $d=3$. Let $\{\widetilde{z}_{n}\}$ be the same sequence as in \eqref{22/2/23/6:1}.  
 Then, by \eqref{21/4/5/10:16} and \eqref{21/4/5/10:25}, we see that for each 
 $0< \varepsilon <1$,   
\begin{equation}\label{21/4/26/11:25}
\lim_{n\to \infty}
\delta(s_{n}) \tau_{n} \langle V \Lambda W, \Lambda W  \rangle   
\le 
-
\frac{1-\varepsilon}{2} A_{1}
.
\end{equation}
Then, by \eqref{21/10/16/17:58},  \eqref{22/2/23/6:1}, \eqref{21/4/9/10:48} and \eqref{21/4/26/11:25}, we see that    
\begin{equation}\label{proof-eq-2}
\begin{split}
&c_{\infty}
\lim_{n \to \infty} 
\frac{s_{n}}{t_{n}}
\langle Q_{n}, \widetilde{z}_{n} \rangle 
\\[6pt]
&=
c_{\infty}
\lim_{n\to \infty}
\frac{1}{A_{1}}\{ 1 +o_{n}(1) \} 
\big\{
\delta(s_{n})
\tau_{n}
\langle  V\Lambda W, \widetilde{z}_{n} \rangle
+
o_{n}(1)
\big\}
\\[6pt]
&= 
c_{\infty}
\lim_{n\to \infty}
\frac{1}{A_{1}}
\delta(s_{n})
\tau_{n}
c_{\infty} \langle  V\Lambda W, \Lambda W \rangle
\le 
- \frac{1-\varepsilon}{2} c_{\infty}^{2}
.
\end{split}
\end{equation}
Furthermore, by \eqref{21/4/8/11:31}, \eqref{proof-eq-2}, $p_{2}<\frac{d+2}{d-2}$ and $c_{\infty}\neq 0$, we see that 
\begin{equation}\label{22/2/19/14:50}
\frac{(d+2)-(d-2)p_{2}}{4} 
\ge 
\frac{1-\varepsilon}{2}
.
\end{equation}
However, taking sufficiently small $\varepsilon$, we see that \eqref{22/2/19/14:50} contradicts $p_{2}>3$. 
 Thus, the claim {\it (i)} of Theorem \ref{20/8/17/11:39} must be true in the case $d=3$. 

Next, assume $d=4$. Then, by \eqref{21/10/16/17:58},  \eqref{22/2/23/6:1}, 
 \eqref{21/4/9/10:48} and \eqref{22/2/22/16:31}, we see that 
\begin{equation}\label{22/2/23/11:34}
\lim_{n \to \infty} 
\frac{s_{n}}{t_{n}}
\langle Q_{n},\widetilde{z}_{n} \rangle 
=-c_{\infty} 
.
\end{equation}
Furthermore, by \eqref{21/4/8/11:31}, \eqref{22/2/23/11:34} and $c_{\infty}\neq 0$, we see that 
\begin{equation}\label{22/2/23/14:19}
\frac{(d+2)-(d-2)p_{2}}{4} 
=
1
.
\end{equation}
However, \eqref{22/2/23/14:19} contradicts $p_{2}>1$. Thus, the claim {\it (i)} of Theorem \ref{20/8/17/11:39} must be true in the case $d=4$. 


\vspace{6pt}


It remains to prove the above claims. 

\noindent 
{\bf Preliminaries.}~By \eqref{21/12/24/13:39}, $\mu_{n}=1+o_{n}(1)$, $\|Q_{n}\|_{L^{\infty}}\lesssim 1$ (see \eqref{19/01/01/12:18}), and \eqref{21/10/14/15:22} in Lemma \ref{proposition:2.3}, we see that 
\begin{equation}\label{22/1/3/17:15}
\begin{split}
\{\mu_{n}M_{\omega_{n}}\}^{-\frac{4}{d-2}}
| g'(\mu_{n}M_{\omega_{n}} Q_{n}) |
&\lesssim 
M_{\omega_{n}}^{-\frac{4}{d-2}}
\big\{
(M_{\omega_{n}}Q_{n})^{p_{1}-1}
+
(M_{\omega_{n}}Q_{n})^{p_{2}-1}
\big\}
\\[6pt]
&\lesssim 
M_{\omega_{n}}^{p_{2}-\frac{d+2}{d-2}}Q_{n}^{p_{1}-1}
=
o_{n}(1).
\end{split} 
\end{equation}

By the Pohozaev's identity \eqref{18/11/24/17:41} (see also the first line of \eqref{20/12/16/16:48}),  \eqref{21/12/24/13:39}, 
 $M_{\omega_{n}}^{-\frac{d+2}{d-2}+p_{2}} \sim t_{n}$ (see \eqref{22/1/3/11:23} and \eqref{22/1/3/14:5}), \eqref{19/01/01/12:18} and \eqref{21/4/28/11:48}, we see that 
\begin{equation}\label{21/3/31/10:3}
\| Q_{n} \|_{L^{2}}^{2} 
\lesssim 
s_{n}^{-1}
M_{\omega_{n}}^{-2^{*}}
\sum_{j=1}^{2} M_{\omega_{n}}^{p_{j}+1}
\| Q_{n} \|_{L^{p_{j}+1}}^{p_{j}+1} 
\lesssim 
s_{n}^{-1}t_{n}
\lesssim 
\delta(s_{n})^{-1}
.
\end{equation}
Moreover, by \eqref{22/1/3/14:5}, \eqref{21/12/24/13:39} and $\|Q_{n}\|_{L^{\infty}}\lesssim 1$ (see \eqref{19/01/01/12:18}), we see that 
\begin{equation}\label{22/1/9/17:50}
t_{n}\kappa_{n}^{-1} | g(\mu_{n}M_{\omega_{n}} Q_{n}) | 
\lesssim 
t_{n} 
\{ M_{\omega_{n}}^{-(p_{2}-p_{1})} 
Q_{n}^{p_{1}} +Q_{n}^{p_{2}}\}
\lesssim 
t_{n}Q_{n}^{p_{1}}
.
\end{equation} 

We shall show that for $d=3,4$, 
\begin{align}\label{21/5/12/9:35}
s_{n} \|Q_{n}\|_{L^{1}}
&\lesssim 1,
\\[6pt]
\label{22/2/24/18:55}
\| \Delta Q_{n} \|_{L^{1}}
&\lesssim 1 
. 
\end{align}
By \eqref{22/1/9/17:50}, H\"older's inequality, \eqref{19/01/01/12:18},
 \eqref{21/3/31/10:3}, $t_{n}\sim \delta(s_{n})^{-1} s_{n}$ (see \eqref{21/4/28/11:48}), and \eqref{19/01/27/16:58}, we see that: if $1<p_{1}\le 2$, then  
\begin{equation}\label{22/3/4/6:52}
\begin{split}
&
\|Q_{n}\|_{L^{\frac{d+2}{d-2}}}^{\frac{d+2}{d-2}} 
+
\int_{\mathbb{R}^{d}}
t_{n} \kappa_{n}^{-1}\big| g(\mu_{n}M_{\omega_{n}}Q_{n}) \big|
\\[6pt]
&\lesssim  
\| Q_{n} \|_{L^{\frac{d+2}{d-2}}}^{\frac{d+2}{d-2}}
+t_{n}
\| Q_{n}\|_{L^{p_{1}}}^{p_{1}} 
\lesssim 
1 + t_{n}\| Q_{n}\|_{L^{1}}^{2-p_{1}} 
\| Q_{n}\|_{L^{2}}^{2(p_{1}-1)}
\\[6pt]
&\lesssim 
1+ t_{n}
\| Q_{n}\|_{L^{1}}^{2-p_{1}} 
\delta(s_{n})^{-(p_{1}-1)}
\lesssim 
1+ s_{n}\delta(s_{n})^{-p_{1}}
\| Q_{n}\|_{L^{1}}^{2-p_{1}} 
\lesssim 
1+ \| Q_{n}\|_{L^{1}}^{2-p_{1}} 
,
\end{split}
\end{equation}
whereas if $2<p_{1}<\frac{d+2}{d-2}$, then  
\begin{equation}\label{22/2/23/18:26}
\begin{split} 
&
\|Q_{n}\|_{L^{\frac{d+2}{d-2}}}^{\frac{d+2}{d-2}} 
+
\int_{\mathbb{R}^{d}}
t_{n} \kappa_{n}^{-1}\big| g(\mu_{n}M_{\omega_{n}}Q_{n}) \big|
\\[6pt]
&\lesssim 
\|Q_{n} \|_{L^{\frac{d+2}{d-2}}}^{\frac{d+2}{d-2}}
+
t_{n}
\| Q_{n}\|_{L^{2}}^{2} \|Q_{n}\|_{L^{\infty}}^{p_{2}-2}
\lesssim 
1 +
\delta(s_{n})^{-2}s_{n} \lesssim 1
.
\end{split} 
\end{equation}
Moreover, by the positivity of $Q_{n}$, \eqref{21/5/12/5:26}, the divergence theorem, and the exponential decay of $|\nabla Q_{n}|$ (see, e.g., Lemma 2 of \cite{Berestycki-Lions}), we see that  
\begin{equation}\label{21/5/12/6:50}
\begin{split}
&s_{n} \| Q_{n}\|_{L^{1}}
=
\int_{\mathbb{R}^{d}} s_{n}Q_{n} 
=
\int_{\mathbb{R}^{d}}
\Delta Q_{n}  
+
\int_{\mathbb{R}^{d}}
Q_{n}^{\frac{d+2}{d-2}} 
+
\int_{\mathbb{R}^{d}}
t_{n} \kappa_{n}^{-1}g(\mu_{n}M_{\omega_{n}}Q_{n}) 
\\[6pt]
&= 
\lim_{R\to \infty}
\int_{|x|=R}
\nabla Q_{n} 
\cdot \frac{x}{|x|}\,d\sigma 
+
\| Q_{n} \|_{L^{\frac{d+2}{d-2}}}^{\frac{d+2}{d-2}}
+
\int_{\mathbb{R}^{d}}
t_{n} \kappa_{n}^{-1}g(\mu_{n}M_{\omega_{n}}Q_{n})
\\[6pt]
&\le
\| Q_{n} \|_{L^{\frac{d+2}{d-2}}}^{\frac{d+2}{d-2}}
+
\int_{\mathbb{R}^{d}}
t_{n} \kappa_{n}^{-1} 
| g(\mu_{n}M_{\omega_{n}}Q_{n}) |
.
\end{split}
\end{equation}
Then, \eqref{21/5/12/9:35} follows from \eqref{22/3/4/6:52},  \eqref{22/2/23/18:26} and \eqref{21/5/12/6:50}. 
 Furthermore, \eqref{22/2/24/18:55} follows from \eqref{21/5/12/9:35}, 
\eqref{21/5/12/5:26}, \eqref{22/3/4/6:52} and \eqref{22/2/23/18:26}.  


We shall show that 
\begin{equation}\label{21/1/27/11}
\|\nabla \widetilde{z}_{n}\|_{L^{2}}\lesssim 1
.
\end{equation}
Multiplying both sides of \eqref{eq:4.30} by $\widetilde{z}_{n}$, integrating the resulting equation, and using the first inequality in \eqref{22/1/3/17:15}, we see that 
\begin{equation}\label{21/1/30/14:37}
\begin{split}
&\|\nabla \widetilde{z}_{n}\|_{L^{2}}^{2}  
+
s_{n}  \| \widetilde{z}_{n}\|_{L^{2}}^{2} 
\\[6pt]
&\lesssim  
\int_{\mathbb{R}^{d}} 
Q_{n}^{\frac{4}{d-2}} 		
| \widetilde{z}_{n} |^{2}
+
\int_{\mathbb{R}^{d}}
\big\{ 
M_{n}^{p_{1}-\frac{d+2}{d-2}} 
Q_{n}^{p_{1}-1} 
 + 
M_{n}^{p_{2}-\frac{d+2}{d-2}}
Q_{n}^{p_{2}-1} 
\big\}
|\widetilde{z}_{n} |^{2} 
.
\end{split} 
\end{equation}
Consider the first term on the right-hand of \eqref{21/1/30/14:37}. 
 When $d=3$, it follows from \eqref{19/01/01/12:18} and \eqref{18/09/12/07:19} that  
\begin{equation}\label{21/1/30/14:48} 
\int_{\mathbb{R}^{3}} 
Q_{n}^{\frac{4}{d-2}} 		
| \widetilde{z}_{n} |^{2}
\le 
\| Q_{n}^{4} \|_{L^{1}}
\| \widetilde{z}_{n} \|_{L^{\infty}}^{2} 
\lesssim 1
,
\end{equation}
whereas when $d=4$, it follows from \eqref{19/01/01/12:18}, 
 \eqref{18/09/12/07:19} and Sobolev's inequality that  
\begin{equation}\label{22/3/11/11:27} 
\int_{\mathbb{R}^{4}} 
Q_{n}^{\frac{4}{d-2}} 		
| \widetilde{z}_{n} |^{2}
\le 
\| Q_{n}^{2} \|_{L^{\frac{4}{3}}}
\| \widetilde{z}_{n}  \|_{L^{\infty}} 
\|  \widetilde{z}_{n}  \|_{L^{4}} 
\lesssim
\| \nabla \widetilde{z}_{n}\|_{L^{2}} 
.
\end{equation}
Consider the second term on the right-hand of \eqref{21/1/30/14:37}. 
 Assume $d=4$. Note that $1<p_{1}<p_{2}<3$. 
 By $M_{\omega_{n}}^{-\frac{4}{d-2}} \sim \omega_{n}^{-1}s_{n}$ (see \eqref{18/12/08/12:02}), \eqref{19/01/01/12:18}, 
 the Gagliardo-Nirenberg inequality and Young's one, 
 the following holds for all $1<q<3$:  
\begin{equation}\label{22/3/11/10:40}
\begin{split}
&
M_{\omega_{n}}^{q-\frac{d+2}{d-2}}
\int_{\mathbb{R}^{4}} 
Q_{n}^{q-1} |\widetilde{z}_{n} |^{2}
\lesssim   
( \omega_{n}^{-1} s_{n})^{\frac{3-q}{2}}
\| Q_{n}\|_{L^{q+1}}^{q-1}
\| \widetilde{z}_{n} \|_{L^{q+1}}^{2}
\\[6pt]
&\lesssim  
o_{n}(1) s_{n}^{\frac{3-q}{2}}
\| \widetilde{z}_{n} \|_{L^{2}}^{\frac{2(3-q)}{q+1}}  
\| \nabla \widetilde{z}_{n} \|_{L^{2}}^{\frac{4(q-1)}{q+1}}
\lesssim 
o_{n}(1) \big\{ s_{n}^{\frac{q+1}{2}}\| \widetilde{z}_{n} \|_{L^{2}}^{2}
+
\| \nabla \widetilde{z}_{n} \|_{L^{2}}^{2}
\big\} 
,
\end{split}
\end{equation}
where the implicit constants may depend on $q$. 
Next, assume $d=3$. Note that $2<p_{1}<p_{2}<5$. 
 By $M_{\omega_{n}}^{-\frac{4}{d-2}} \sim \omega_{n}^{-1}s_{n}$, H\"older's inequality, \eqref{19/01/01/12:18} and Sobolev's inequality, we see that:  
for $3< q<5$,     
\begin{equation}\label{21/1/30/14:47}
\begin{split}
M_{n}^{q-\frac{d+2}{d-2}}\int_{\mathbb{R}^{3}} 
Q_{n}^{q-1} |\widetilde{z}_{n} |^{2} 
&\lesssim  
(\omega_{n}^{-1} s_{n})^{\frac{5-q}{4}}
\|Q_{n}^{q-1}\|_{L^{\frac{3}{2}}}
\| |\widetilde{z}_{n}|^{2}\|_{L^{3}} 
\lesssim
o_{n}(1)
\| \nabla \widetilde{z}_{n} \|_{L^{2}}^{2}  
,
\end{split}
\end{equation}
and for $2< q\le 3$,   
\begin{equation}\label{22/3/11/10}
\begin{split}
&M_{\omega_{n}}^{q-\frac{d+2}{d-2}}\int_{\mathbb{R}^{3}} 
Q_{n}^{q-1} |\widetilde{z}_{n} |^{2} 
\lesssim  
(\omega_{n}^{-1}s_{n})^{\frac{5-q}{4}}
\|Q_{n}^{q-1}\|_{L^{3}}
\| \widetilde{z}_{n} \|_{L^{2}}
\| \widetilde{z}_{n} \|_{L^{6}}
\\[6pt]
&\lesssim 
o_{n}(1) \big\{ 
s_{n}^{\frac{5-q}{2}}
\| \widetilde{z}_{n} \|_{L^{2}}^{2}
+
\| \nabla \widetilde{z}_{n} \|_{L^{2}}^{2}
\big\}
\le  
o_{n}(1) \big\{ 
s_{n}\| \widetilde{z}_{n} \|_{L^{2}}^{2}
+
\| \nabla \widetilde{z}_{n} \|_{L^{2}}^{2}
\big\},
\end{split}
\end{equation}
where the implicit constants in \eqref{21/1/30/14:47} and \eqref{22/3/11/10} may depend on $q$. 
Plugging the estimates  \eqref{21/1/30/14:48} through \eqref{22/3/11/10} into \eqref{21/1/30/14:37}, 
 we obtain \eqref{21/1/27/11}.  

We can prove the following uniform estimates in a way similar to Lemma 4.3 of \cite{AIIKN}:
\begin{equation}\label{19/01/01/08:23}
|\widetilde{z}_{n}(x)| \lesssim  (1+|x|)^{-(d-2)},
\qquad 
\|\widetilde{z}_{n}\|_{L^{r}} \lesssim 1
\quad 
\mbox{for all $\frac{d}{d-2}<r\le \infty$}
.  
\end{equation}


Now, we give the proofs of the above claims: 

\noindent 
{\bf Proof of Claim 1.}~We shall prove \eqref{22/2/23/6:1}.

By \eqref{19/01/01/08:23}, \eqref{eq:4.30}, \eqref{22/1/3/17:15} 
 and \eqref{19/01/01/12:18}, we see that 
\begin{equation}\label{22/1/4/15:3}
\| \widetilde{z}_{n}\|_{W^{2,2d}}
\sim 
\|\widetilde{z}_{n}\|_{L^{2d}}
+
\|\Delta \widetilde{z}_{n}\|_{L^{2d}} \lesssim 1
.
\end{equation}
Then, it follows from $\widetilde{z}_{n}$ being radial, \eqref{21/1/27/11}, \eqref{22/1/4/15:3} and the Rellich-Kondrashov compactness theorem that there exist a subsequence of $\{\widetilde{z}_{n}\}$ (which we continue to denote by $\{\widetilde{z}_{n}\}$) and a real-valued radial function $\widetilde{z}_{\infty} \in \dot{H}_{\rm{rad}}^{1}(\mathbb{R}^{d}) \cap C_{\rm{loc}}^{1}(\mathbb{R}^{d})$  such that 
\begin{equation}\label{18/09/11/20:31} 
\lim_{n\to \infty}\widetilde{z}_{n}
=\widetilde{z}_{\infty} 
\quad \mbox{weakly in $\dot{H}^{1}(\mathbb{R}^{d})$ and strongly in $C_{\rm{loc}}^{1}(\mathbb{R}^{d})$}
.
\end{equation}
Furthermore, by \eqref{19/01/01/08:23}, we may assume that 
\begin{equation}\label{22/1/8/13:57} 
\lim_{n\to \infty}\widetilde{z}_{n}
=\widetilde{z}_{\infty} 
\quad 
\mbox{weakly in $L^{q}(\mathbb{R}^{d})$ for all $\frac{d}{d-2}< q <\infty$}
.
\end{equation}

Now, let $\phi \in C_{c}^{\infty}(\mathbb{R}^{d})$ be a test function. 
Then, by \eqref{eq:4.30}, the fundamental theorem of calculus, 
 \eqref{18/09/11/20:31}, \eqref{20/12/28/1}, 
 \eqref{18/09/12/07:19}, \eqref{22/1/3/17:15}, \eqref{19/01/01/12:18} and 
 \eqref{20/12/28/10:26}, we see that 
\begin{equation}\label{20/12/24/17:39}
\begin{split}
&
\big|
\langle  (-\Delta + V)  \widetilde{z}_{\infty}, \phi \rangle 
\big|
\\[6pt]
&\lesssim  
\big|
\langle \Delta  \{\widetilde{z}_{n} -\widetilde{z}_{\infty}\}, \phi \rangle 
\big|
+
s_{n} 
| \langle  \widetilde{z}_{n},\phi \rangle |
+
\{\mu_{n}M_{\omega_{n}}\}^{-\frac{4}{d-2}}
\big| \langle g'(\mu_{n}M_{\omega_{n}}Q_{n}) 
\widetilde{z}_{n}, \phi \rangle 
\big|
\\[6pt]
&\quad +
\big|
\langle   
Q_{n}^{\frac{4}{d-2}} \big\{ \widetilde{z}_{n} -\widetilde{z}_{\infty} \big\},
 \phi \rangle
\big|  
+
\big|
\langle   
\int_{0}^{1} \{  W+ \theta \zeta_{n} \}^{\frac{6-d}{d-2}} \zeta_{n} \, d\theta
\, \widetilde{z}_{\infty},
 \phi \rangle  
\big|
\\[6pt]
&\to 0 
\quad 
\mbox{as $n\to \infty$}
.
\end{split}
\end{equation}
Thus, it holds that $(-\Delta  + V) \widetilde{z}_{\infty}=0$ in the distribution sense. 
 Then, we can verify that $\widetilde{z}_{\infty} \in W^{2,2^{*}}(\mathbb{R}^{d})$, which together with the Sobolev's inequality and $d=3,4$ shows that $\widetilde{z}_{\infty} \in C(\mathbb{R}^{d})\cap L^{\infty}(\mathbb{R}^{d})$, which further implies $\widetilde{z}_{\infty} \in W^{2,q}(\mathbb{R}^{d})$ for all $q>2^{*}$. 
 Then, the Sobolev's inequality shows $\widetilde{z}_{\infty} \in C^{1}(\mathbb{R}^{d})$. Furthermore, by the same argument as in the proof of Lemma 1 of \cite{Berestycki-Lions}, we see that $\widetilde{z}_{\infty} \in C^{2}(\mathbb{R}^{d})$. 
 Then, by Lemma A.1 of \cite{CMR}, we see that there exists $c_{\infty} \in \mathbb{R}$ such that 
\begin{equation}\label{eq:4.33}
\widetilde{z}_{\infty} = c_{\infty} \Lambda W.  
\end{equation}
Note that the convergence in $C_{\rm{loc}}^{1}(\mathbb{R}^{d})$ (see \eqref{18/09/11/20:31}) together with \eqref{18/09/12/07:19} and the decay estimate in \eqref{19/01/01/08:23} shows that $c_{\infty} \neq 0$.  


\noindent 
{\bf Proof of Claim 2.}~We shall prove \eqref{proof-eq-0}.

Put  $\widetilde{w}_{n} := x \cdot \nabla Q_{n}$.  
 By an elementary computation, \eqref{21/5/12/5:26} and $t_{n}\kappa_{n}^{-1}=\{\mu_{n}M_{\omega_{n}}\}^{-\frac{d+2}{d-2}}$ (see \eqref{22/1/3/11:23}), 
 we can verify that  
\begin{equation}\label{eq:4.34}
\begin{split}
-\Delta \widetilde{w}_{n} 
+ 
s_{n} \widetilde{w}_{n} 
&=
x\cdot \nabla 
\{ -\Delta Q_{n} + s_{n}Q_{n} \}
-2 \Delta Q_{n} 
\\[6pt]
&= 
\Big\{  
\frac{d+2}{d-2} Q_{n}^{\frac{4}{d-2}}
+
\{ \mu_{n} M_{\omega_{n}}\}^{-\frac{4}{d-2}} g'(\mu_{n} M_{\omega_{n}}Q_{n})  
\Big\}
\widetilde{w}_{n}
\\[6pt] 
&\quad  -
2 \big\{ s_{n} Q_{n}  -
Q_{n}^{\frac{d+2}{d-2}} 
-
t_{n} \kappa_{n}^{-1} 
g(\mu_{n} M_{\omega_{n}} Q_{n})  
\big\}
.
\end{split} 
\end{equation}
Multiplying both sides of \eqref{eq:4.30} by $\widetilde{w}_{n}$, and \eqref{eq:4.34} by $\widetilde{z}_{n}$, and integrating the difference of the resulting equations, we see that  
\begin{equation}\label{eq:4.35}
s_{n} 
\langle Q_{n},  \widetilde{z}_{n} \rangle 
=
\langle 
Q_{n}^{\frac{d+2}{d-2}}, \widetilde{z}_{n}
\rangle
+
t_{n} \kappa_{n}^{-1}
\langle 
g(\mu_{n} M_{\omega_{n}} Q_{n}), \widetilde{z}_{n}  
\rangle
.
\end{equation}
Moreover, multiplying both sides of \eqref{21/5/12/5:26} by $\widetilde{z}_{n}$,  and \eqref{eq:4.30} by $Q_{n}$, integrating the difference of the resulting equations, and using $t_{n}\kappa_{n}^{-1}=\{\mu_{n}M_{\omega_{n}}\}^{-\frac{d+2}{d-2}}$, we see that    
\begin{equation}\label{18/09/11/21:07}
\begin{split}
\frac{4}{d-2}
\langle 
Q_{n}^{\frac{d+2}{d-2}},  \widetilde{z}_{n} 
\rangle 
&= 
t_{n} \kappa_{n}^{-1}
\langle  
g(\mu_{n} M_{\omega_{n}} Q_{n})
-
\mu_{n} M_{\omega_{n}}g'(\mu_{n} M_{\omega_{n}} Q_{n}) Q_{n}, \widetilde{z}_{n}
\rangle
\\[6pt]
&=
\frac{4}{d-2} t_{n} \kappa_{n}^{-1} \{ \mu_{n} M_{\omega_{n}}\}^{p_{2}}
\rho_{n}
-
\frac{4}{d-2} t_{n} \kappa_{n}^{-1} \langle g(\mu_{n} M_{\omega_{n}} Q_{n}), \widetilde{z}_{n}
\rangle. 
\end{split}
\end{equation}
Plugging \eqref{18/09/11/21:07} divided by $\frac{4}{d-2}$ into \eqref{eq:4.35}, 
 we obtain \eqref{proof-eq-0}. 


\noindent 
{\bf Proof of Claim 3.}~We shall prove \eqref{21/4/8/11:31}. To this end, it suffices to show that 
\begin{equation}\label{proof-eq-1}
\lim_{n \to \infty} 
\rho_{n}
=
\frac{(d+2)-(d-2)p_{2}}{4}\frac{C_{2}}{p_{2}}  \langle W^{p_{2}}, \Lambda W \rangle c_{\infty}
.
\end{equation}
Indeed, it follows from \eqref{proof-eq-0} and \eqref{22/1/3/14:5} that 
\begin{equation}\label{22/3/5/10:55}
c_{\infty}\frac{s_{n}}{t_{n}}
\lim_{n\to \infty}
\langle Q_{n} , \widetilde{z}_{n} \rangle  
= 
c_{\infty}
\lim_{n\to \infty}
\frac{\{\mu_{n}M_{\omega_{n}}\}^{p_{2}}}{\kappa_{n}} 
\lim_{n\to \infty} \rho_{n} 
=
-\frac{c_{\infty}p_{2}}{C_{2} \langle W^{p_{2}}, \Lambda W \rangle}
\lim_{n\to \infty} \rho_{n}
,
\end{equation}
which together with \eqref{proof-eq-1} implies \eqref{21/4/8/11:31}.

Let us prove \eqref{proof-eq-1}. Observe from the definition of $\rho_{n}$ (see \eqref{22/1/5/15:41}) that 
\begin{equation}\label{20/12/26/12:16}
\begin{split}
&
\Big| 
\rho_{n}- \frac{(d+2)-(d-2)p_{2}}{4}\frac{C_{2}}{p_{2}} \langle W^{p_{2}}, \Lambda W \rangle c_{\infty}
\Big|
\\[6pt]
&\lesssim 
\Big|  \langle 
 \frac{g(\mu_{n}M_{\omega_{n}}Q_{n})}{ (\mu_{n}M_{\omega_{n}})^{p_{2}}}
, \widetilde{z}_{n} \rangle
-
\frac{C_{2}}{p_{2}} \langle W^{p_{2}},\widetilde{z}_{n}  \rangle 
-
\frac{C_{2}}{p_{2}} \langle W^{p_{2}}, c_{\infty}\Lambda W -\widetilde{z}_{n} 
\rangle \Big| 
\\[6pt]
&\quad +
\Big|
\langle \frac{g'(\mu_{n}M_{\omega_{n}}Q_{n})}{ (\mu_{n}M_{\omega_{n}})^{p_{2}-1}} Q_{n}
, \widetilde{z}_{n} \rangle
-
\frac{C_{2}}{p_{2}} \langle W^{p_{2}}, \widetilde{z}_{n}  \rangle 
-
\frac{C_{2}}{p_{2}} \langle W^{p_{2}}, c_{\infty}\Lambda W -\widetilde{z}_{n} 
\rangle
\Big| 
\\[6pt]
&\lesssim  
| \langle W^{p_{2}}, \widetilde{z}_{n}-c_{\infty}\Lambda W  
\rangle |
+
\Big| 
\langle 
\frac{g(\mu_{n}M_{\omega_{n}}Q_{n})}{(\mu_{n} M_{\omega_{n}})^{p_{2}}}
-
\frac{C_{2}}{p_{2}}W^{p_{2}}, \widetilde{z}_{n} 
\rangle 
\Big|
\\[6pt]
&\quad +
\Big| \langle 
\frac{g'(\mu_{n}M_{\omega_{n}}Q_{n})}{(\mu_{n} M_{\omega_{n}})^{p_{2}-1}} Q_{n}
-
C_{2}W^{p_{2}}
,\widetilde{z}_{n} \rangle 
\Big| 
.
\end{split}
\end{equation}
Consider the first term on the right-hand side of \eqref{20/12/26/12:16}. 
 Note that for any $0<\varepsilon < \min\{ (d-2)p_{2}-2, d-2 \}$, it holds that $\dfrac{dp_{2}}{2+\varepsilon} > \dfrac{d}{d-2}$ and $W^{p_{2}}\in L^{\frac{d}{2+\varepsilon}}(\mathbb{R}^{d})$. Then, we see from \eqref{22/1/8/13:57} and \eqref{eq:4.33} that 
\begin{equation}\label{20/12/26/12:19}
\lim_{n\to \infty}
| \langle  
W^{p_{2}}, \widetilde{z}_{n}-c_{\infty}\Lambda W  
\rangle |
=0
.
\end{equation}
Consider the second term on the right-hand side of \eqref{20/12/26/12:16}. 
 We shall show that 
\begin{equation}\label{22/1/9/12:6}
\lim_{n\to \infty}
\Big| 
\langle 
\frac{g(\mu_{n}M_{\omega_{n}}Q_{n})}{(\mu_{n} M_{\omega_{n}})^{p_{2}}} 
-
\frac{C_{2}}{p_{2}}W^{p_{2}}
,\widetilde{z}_{n} \rangle 
\Big| 
=0
.
\end{equation}
First, observe that 
\begin{equation}\label{22/1/8/16:25}
\begin{split} 
&
\Big| 
\langle 
\frac{
g(\mu_{n}M_{\omega_{n}}Q_{n}) }{(\mu_{n} M_{\omega_{n}})^{p_{2}}}
-
\frac{C_{2}}{p_{2}}W^{p_{2}}
,\widetilde{z}_{n} \rangle 
\Big| 
\\[6pt]
&\le 
\int_{\mathbb{R}^{d}}
\Big|
\frac{g(\mu_{n}M_{\omega_{n}}Q_{n})
}{\{ \mu_{n}M_{\omega_{n}}Q_{n}\}^{p_{2}}} 
-
\frac{C_{2}}{p_{2}}
\Big| 
Q_{n}^{p_{2}} | \widetilde{z}_{n} |
\,dx 
+
\frac{C_{2}}{p_{2}}
| \langle 
Q_{n}^{p_{2}}  
-
W^{p_{2}}
, \widetilde{z}_{n} \rangle |
.
\end{split}
\end{equation}
By \eqref{21/12/25/15:11}, $Q_{n}$ being positive,  \eqref{21/10/14/15:22}, Lemma \ref{theorem:3.1}, $\|Q_{n}\|_{L^{\infty}}\lesssim 1$ and 
 $\|\widetilde{z}_{n}\|_{L^{\infty}}= 1$ (see \eqref{18/09/12/07:19}), 
 we see that 
\begin{equation}\label{22/1/8/18:25}
\lim_{n\to \infty}
\Big|
\frac{g(\mu_{n}M_{\omega_{n}}Q_{n})
}{\{ \mu_{n}M_{\omega_{n}}Q_{n}\}^{p_{2}}} 
-
\frac{C_{2}}{p_{2}}
\Big| Q_{n}^{p_{2}} | \widetilde{z}_{n} |=0
\qquad 
\mbox{almost everywhere in $\mathbb{R}^{d}$}
. 
\end{equation}
Furthermore, by \eqref{21/12/24/13:39}, \eqref{19/01/01/12:18}, \eqref{19/01/01/08:23} and $\frac{2}{d-2}<p_{1}< p_{2}$, we see  that 
\begin{equation}\label{22/1/8/18:27}
\begin{split}
\Big|
\frac{g(\mu_{n}M_{\omega_{n}}Q_{n})
}{\{ \mu_{n}M_{\omega_{n}}Q_{n}\}^{p_{2}}} 
-
\frac{C_{2}}{p_{2}}
\Big| 
Q_{n}^{p_{2}} | \widetilde{z}_{n} |
&\lesssim 
\{Q_{n}^{p_{1}}+ Q_{n}^{p_{2}}\} | \widetilde{z}_{n} |
\\[6pt]
&\lesssim  
(1+|x|)^{-(d-2)(p_{1}+1)}
\in L^{1}(\mathbb{R}^{d})
.
\end{split} 
\end{equation} 
Hence, the Lebesgue's dominated convergence theorem shows that 
 the first term on the right-hand side of \eqref{22/1/8/16:25} obeys 
\begin{equation}\label{18/12/08/16:59}
\lim_{n\to \infty}
\int_{\mathbb{R}^{d}}
\Big|
\frac{g(\mu_{n}M_{\omega_{n}}Q_{n})
}{\{ \mu_{n}M_{\omega_{n}}Q_{n}\}^{p_{2}}} 
-
\frac{C_{2}}{p_{2}}
\Big| 
Q_{n}^{p_{2}} | \widetilde{z}_{n} |
\,dx =0.
\end{equation}
Move on to the second term on the right-hand side of \eqref{22/1/8/16:25}.
 Fix $0<\varepsilon < 1$ with $2\varepsilon <(d-2)p_{2}-2 $. Note that $\frac{d(p_{2}-1)}{4-d+2\varepsilon} > \frac{d}{d-2}$. 
 Then, by the fundamental theorem of calculus, 
 H\"older's inequality, \eqref{20/12/28/10:26} and \eqref{19/01/01/08:23}, we see that 
\begin{equation}\label{22/1/9/12:9}
\begin{split}
\lim_{n\to \infty}
| \langle 
Q_{n}^{p_{2}}  
-
W^{p_{2}}
, \widetilde{z}_{n} \rangle |
&\lesssim  
\lim_{n\to \infty}
\|W+|\zeta_{n}| \|_{L^{\frac{d(p_{2}-1)}{4-d+2\varepsilon}}}^{p_{2}-1}
\| \zeta_{n} \|_{L^{\frac{d}{d-2-\varepsilon}}}
\| \widetilde{z}_{n} \|_{L^{\frac{d}{d-2-\varepsilon}}}
=0
.
\end{split} 
\end{equation}
Then, \eqref{22/1/9/12:6} follows from  \eqref{22/1/8/16:25}, \eqref{18/12/08/16:59} and \eqref{22/1/9/12:9}. Similarly, we can prove that the last term on the right-hand side of \eqref{20/12/26/12:16} obeys 
\begin{equation}\label{22/1/9/14:53}
\lim_{n\to \infty}
\Big| \langle 
\frac{g'(\mu_{n}M_{\omega_{n}}Q_{n})}{(\mu_{n} M_{\omega_{n}})^{p_{2}-1}} 
Q_{n}
-
C_{2}W^{p_{2}}
,\widetilde{z}_{n} \rangle 
\Big| 
=0
.
\end{equation}
Then, \eqref{proof-eq-1} follows from \eqref{20/12/26/12:16}, \eqref{20/12/26/12:19}, \eqref{22/1/9/12:6} and \eqref{22/1/9/14:53}.  


\noindent 
{\bf Proof of Claim 4.}~We shall prove \eqref{21/4/5/10:16}.
  By  $V \Lambda W = -(-\Delta) \Lambda W$, 
 the second resolvent equation and Lemma \ref{21/1/10/15:37}, 
 we see that 
\begin{equation}\label{21/4/5/10:17}
\begin{split}
&
|
\langle  (-\Delta +s_{n})^{-1} V \Lambda W, V\Lambda W \rangle
+
\langle \Lambda W, V\Lambda W \rangle  | 
\\[6pt]
&=
|
\langle  (-\Delta +s_{n})^{-1} V \Lambda W,V\Lambda W \rangle
-
\langle (-\Delta)^{-1} V\Lambda W,V\Lambda W \rangle  
| 
\\[6pt]
&= 
s_{n}  | \langle (-\Delta +s_{n})^{-1}  (-\Delta)^{-1} V\Lambda W, V\Lambda W
\rangle |
=
s_{n} | 
\langle (-\Delta +s_{n})^{-1}\Lambda W,V\Lambda W
\rangle |
\\[6pt]
&\le 
s_{n}   \| (-\Delta +s_{n})^{-1}  \Lambda W\|_{L^{\infty}}
\| V\Lambda W \|_{L^{1}}
\lesssim  
s_{n} \delta(s_{n} )^{-1}  
,
\end{split}
\end{equation}
which together with $V:=-\frac{d+2}{d-2}W^{\frac{4}{d-2}}<0$ (see \eqref{20/9/13/9:52}) implies \eqref{21/4/5/10:16}.


\noindent 
{\bf Proof of Claim 5.}~We shall prove \eqref{21/4/5/10:25}. To this end, it suffices to show that for each $0< \varepsilon<1$, 
\begin{equation}\label{22/2/8/2:8}
\liminf_{n\to \infty}\delta(s_{n})\{(-\Delta+s_{n})^{-1}Q_{n}\}(0)
\ge \frac{1-\varepsilon}{2} A_{0}
.
\end{equation}
Indeed, by Lemma \ref{22/1/25/7:1}, \eqref{19/01/01/12:18}, \eqref{22/2/8/2:8} and $A_{1}:=\mathscr{C}_{0}A_{0}$ (see \eqref{20/12/22/13:58}), we see that 
\begin{equation}\label{22/3/5/10}
\begin{split}
\delta(s_{n}) \langle 
(-\Delta +s)^{-1} Q_{n}, V\Lambda W \rangle 
&=
\delta(s_{n})
\mathscr{C}_{0} \{(-\Delta +s)^{-1}Q_{n}\}(0)
+ 
\delta(s_{n})
\| Q_{n}\|_{L^{2^{*}}}
\\[6pt]
&\ge 
\mathscr{C}_{0} \frac{1-\varepsilon}{2} A_{0}
+o_{n}(1)
=
\frac{1-\varepsilon}{2} A_{1}+o_{n}(1),
\end{split} 
\end{equation}
which implies \eqref{21/4/5/10:25}. Let us prove \eqref{22/2/8/2:8}. Assume $d=3$. Put $Y_{n}:=\sqrt{3}e^{-\sqrt{s_{n}}|x|}|x|^{-1}$. 
 Note that 
\begin{equation}\label{22/3/1/13:23}
\Delta Y_{n} -s_{n}Y_{n} 
=0
\qquad 
\mbox{in $\mathbb{R}^{3}\setminus \{0\}$}
.
\end{equation}
Let $\varepsilon_{0}>0$ and $R>0$. Then, the convergence of $\{\widetilde{\Phi}_{\omega}\}$ to $W$ in $C_{\rm{loc}}^{2}(\mathbb{R}^{d})$ (see Lemma \ref{theorem:3.1}) together with $Q_{n}=T_{\mu_{n}}[ \widetilde{\Phi}_{\omega_{n}}]$ and $\mu_{n}=1+o_{n}$ shows that  there exists a number $n(\varepsilon_{0}, R)$ such that if $n\ge n(\varepsilon_{0}, R)$, then  
\begin{equation}\label{22/3/1/18:8}
Q_{n}(R)
\ge 
(1-\varepsilon_{0}) W(R) 
\ge 
 \frac{(1-\varepsilon_{0})R}{\sqrt{3+R^{2}}} 
Y_{n}(R).
\end{equation}
Put  
\begin{equation}\label{22/3/1/15:41}
\varepsilon :=1-\frac{(1-\varepsilon_{0})R}{\sqrt{3+R^{2}}} >0
.
\end{equation}
Note that taking $\varepsilon_{0}$ sufficiently small and $R$ sufficiently large, we can take $\varepsilon$ as small as we want. 
By \eqref{22/3/1/13:23}, \eqref{21/5/12/5:26}, the positivity of $g$ on $(0,\infty)$ (see \eqref{22/3/3/8:40}) and the definition of $t_{n}$ (see \eqref{22/1/3/11:23}), we see that 
\begin{equation}\label{22/3/1/13:34}
\Delta \{ (1-\varepsilon) Y_{n}-Q_{n}\} -s_{n}\{ (1-\varepsilon)Y_{n}-Q_{n}\} 
=
Q_{n}^{5}
+
\{\mu_{n}M_{n}\}^{-5}g(\mu_{n}M_{\omega_{n}}Q_{n})
\ge 0
.
\end{equation}
Then, the maximum principle together with \eqref{22/3/1/18:8}  and \eqref{22/3/1/13:34} shows that 
\begin{equation}\label{22/3/1/14:42}
\chi_{\ge R}Q_{n}
\ge 
(1-\varepsilon) 
\chi_{\ge R} Y_{n}
\qquad 
\mbox{for all $n\ge n(\varepsilon_{0},R)$},
\end{equation}
where $\chi_{\ge R}$ is the indicator function of $\{|x|\ge R\}$, namely $\chi_{\ge R}(x)=0$ if $|x|\le R$ and $\chi_{\ge R}(x)=1$ if $|x|\ge R$. 
 By \eqref{22/3/1/14:42}, $4\pi e^{-\sqrt{s_{n}}|x|}|x|^{-1}$ being the Green's function for $-\Delta +s_{n}$ in three dimensions (see, e.g., Theorem 6.23 of \cite{Lieb-Loss}), the substitution of variables, and $Q_{n}$ being positive,  we see that if $n\ge n(\varepsilon_{0},R)$, then 
\begin{equation}\label{22/2/8/10:11}
\begin{split}
&
\delta(s_{n}) \{ (-\Delta +s_{n} )^{-1}Q_{n}\}(0)
\\[6pt]
&= 
\delta(s_{n}) \{ (-\Delta +s_{n} )^{-1} (1-\chi_{\ge R})Q_{n}\}(0)
+
\delta(s_{n}) \{ (-\Delta +s_{n} )^{-1} \chi_{\ge  R}Q_{n}\}(0)
\\[6pt]
&\ge 
(1-\varepsilon) 
\delta(s_{n}) \{ (-\Delta +s_{n} )^{-1}\chi_{\ge R}Y_{n} \}(0)
\\[6pt]
&= 
(1-\varepsilon) 
 \delta(s_{n}) \int_{|x|\ge R} 
\frac{e^{-\sqrt{s_{n}}|x|}}{4\pi |x|} \sqrt{3} e^{-\sqrt{s_{n}}|x|}|x|^{-1}\,dx \\[6pt]
&=
(1-\varepsilon) 
\sqrt{3}
\int_{|x|\ge \sqrt{s_{n}}R} \frac{e^{-2|x|}}{4\pi |x|^{2}}
\,dx  
=
(1-\varepsilon) 
\frac{\sqrt{3}}{2}
\int_{|x|\ge 2 \sqrt{s_{n}}R} \frac{e^{-|x|}}{4\pi |x|^{2}}
\,dx  
\end{split}
\end{equation}
It follows from \eqref{22/2/8/10:11} that 
\begin{equation}\label{22/3/1/15:39}
\begin{split} 
\lim_{n\to \infty}
\delta(s_{n}) \{ (-\Delta +s_{n} )^{-1}Q_{n}\}(0)
\ge 
\frac{1-\varepsilon}{2}
\sqrt{3}
\int_{\mathbb{R}^{3}}
\frac{e^{-|x|}}{4\pi |x|^{2}}
\,dx 
=\frac{1-\varepsilon}{2}
A_{0}
.
\end{split} 
\end{equation}
Thus, we have proved \eqref{22/2/8/2:8} and therefore the claim \eqref{21/4/5/10:25} is true.


\noindent 
{\bf Proof of Claim 6.}~We shall prove \eqref{22/2/22/16:31}.  
 By $Q_{n}=W+\zeta_{n}$ and Lemma \ref{20/12/22/13:55}, we see that  
\begin{equation}\label{21/5/6/7:2}
\begin{split}
&
| \delta(s_{n}) \langle  (-\Delta +s_{n})^{-1}
 Q_{n},V \Lambda W \rangle -A_{1} |
\le 
o_{n}(1)
+
| 
\delta(s_{n}) \langle  (-\Delta +s_{n})^{-1}
 \zeta_{n}, V \Lambda W \rangle |
. 
\end{split} 
\end{equation}
Hence, for \eqref{22/2/22/16:31}, it suffices to show that 
\begin{equation}\label{22/2/23/17:18}
\lim_{n\to \infty}
| \delta(s_{n}) \langle  (-\Delta +s_{n})^{-1}
 \zeta_{n}, V \Lambda W \rangle | =0
.
\end{equation}
Assume $d=4$. Let $0< \varepsilon_{1} < \min\{ (d-2)p_{1}-2,1 \}$.  Then, by $V \Lambda W = \Delta \Lambda W$ (see \eqref{20/9/13/9:52}), H\"older's inequality, and \eqref{18/11/24/13:56} in Lemma \ref{18/11/05/10:29},  
 we see that 
\begin{equation}\label{21/5/6/7:3}
\begin{split}
&
\delta(s_{n}) 
| \langle  (-\Delta +s_{n})^{-1} \zeta_{n}, V \Lambda W \rangle |
=
\delta(s_{n}) 
| \langle   
\Delta \zeta_{n},
 (-\Delta +s_{n})^{-1} \Lambda W \rangle |
\\[6pt]
&\le 
\delta(s_{n}) 
\| \Delta \zeta_{n} \|_{L^{\frac{d}{2+\varepsilon_{1}}}}
\| (-\Delta +s_{n})^{-1}  \Lambda W  \|_{L^{\frac{d}{d-2-\varepsilon_{1}}}}
\\[6pt]
&\lesssim
\delta(s_{n}) 
\| \Delta \zeta_{n} \|_{L^{\frac{d}{2+\varepsilon_{1}}}}
s_{n}^{\frac{\varepsilon_{1}}{2}-1}
\| \Lambda W  \|_{L_{\rm{weak}}^{\frac{d}{d-2}}}
\lesssim
\delta(s_{n}) s_{n}^{\frac{\varepsilon_{1}}{2}-1}
\| \Delta \zeta_{n} \|_{L^{\frac{d}{2+\varepsilon_{1}}}}
,
\end{split}
\end{equation}
where the implicit constants may depend on $\varepsilon_{1}$.
Note here that $\frac{dp_{1}}{2+\varepsilon_{1}}>\frac{d}{d-2}$ and 
 $\frac{16d}{(d-2)(8+3 \varepsilon_{1})}>\frac{d}{d-2}$. 
 Hence, by \eqref{21/5/12/5:26}, \eqref{22/1/9/17:50}, the fundamental theorem of calculus, H\"older's inequality,  
 \eqref{19/01/01/12:18}, \eqref{20/12/28/10:26} with $\varepsilon=\frac{\varepsilon_{1}}{8}$, and $t_{n} \sim \delta(s_{n})^{-1} s_{n}$ (see \eqref{21/4/28/11:48}), we see that 
\begin{equation}\label{21/5/12/5:45}
\begin{split}
&
\|\Delta \zeta_{n} \|_{L^{\frac{d}{2+\varepsilon_{1}}}}
=
\|\Delta \{Q_{n} -W\} \|_{L^{\frac{d}{2+\varepsilon_{1}}}}
=
\|\Delta Q_{n} +W^{\frac{d+2}{d-2}} \|_{L^{\frac{d}{2+\varepsilon_{1}}}}
\\[6pt]
&\lesssim 
s_{n}\| Q_{n}\|_{L^{\frac{d}{2+\varepsilon_{1}}}}
+
\| (W+|\zeta_{n}|)^{\frac{4}{d-2}} \zeta_{n}  \|_{L^{\frac{d}{2+\varepsilon_{1}}}}
+
t_{n}\| Q_{n}^{p_{1}} \|_{L^{\frac{d}{2+\varepsilon_{1}}}}
\\[6pt]
&\le 
s_{n}\| Q_{n}\|_{L^{\frac{d}{2+\varepsilon_{1}}}}
+
\| W+|\zeta_{n}| \|_{L^{\frac{16d}{(d-2)(8+ 3\varepsilon_{1})}}}^{\frac{4}{d-2}} 
\|\zeta_{n}\|_{L^{\frac{4d}{\varepsilon_{1}}}}
+
t_{n}\| Q_{n} \|_{L^{\frac{dp_{1}}{2+\varepsilon_{1}}}}^{p_{1}}
\\[6pt]
&\lesssim 
s_{n}\| Q_{n}\|_{L^{\frac{d}{2+\varepsilon_{1}}}}
+
s_{n}^{\frac{d-2}{2}-\frac{\varepsilon_{1}}{4}}
+
\delta(s_{n})^{-1} s_{n}
, 
\end{split}
\end{equation}
where the implicit constants  may depend on $\varepsilon_{1}$. 
 Furthermore, by the interpolation inequality (H\"older's inequality), \eqref{21/3/31/10:3} and \eqref{21/5/12/9:35}, 
 we see that 
\begin{equation}\label{21/5/12/9:45}
s_{n}\| Q_{n}\|_{L^{\frac{d}{2+\varepsilon_{1}}}}
\le 
s_{n}
\| Q_{n}\|_{L^{1}}^{\frac{4-d+2\varepsilon_{1}}{d}}
\| Q_{n}\|_{L^{2}}^{\frac{2(d-2-\varepsilon_{1})}{d}}
\lesssim
s_{n}^{\frac{2(d-2-\varepsilon_{1})}{d}}
\delta(s_{n})^{\frac{-(d-2-\varepsilon_{1})}{d}}
,
\end{equation}
where the implicit constant may depend on $\varepsilon_{1}$. 
 Plugging \eqref{21/5/12/9:45} into \eqref{21/5/12/5:45}, we obtain 
\begin{equation}\label{21/5/12/5:40}
\|\Delta \zeta_{n}\|_{L^{\frac{d}{2+\varepsilon_{1}}}}
\lesssim 
\delta(s_{n})^{\frac{-(d-2-\varepsilon_{1})}{d}}
s_{n}^{\frac{2(d-2-\varepsilon_{1})}{d}}
+
s_{n}^{\frac{d-2}{2}-\frac{\varepsilon_{1}}{4}}
+
\delta(s_{n})^{-1} s_{n}
,
\end{equation}
where the implicit constant may depend on $\varepsilon_{1}$. 
 Then, it follows from \eqref{21/5/6/7:3}, \eqref{21/5/12/5:40}, the definition of $\delta(s)$ (see \eqref{19/01/27/16:58}) and $d=4$ that 
\begin{equation}\label{22/3/5/9:2}
\begin{split}
&
\delta(s_{n}) 
| \langle  (-\Delta +s_{n})^{-1} \zeta_{n}, V \Lambda W \rangle |
\\[6pt]
&\lesssim 
s_{n}^{1-\frac{4}{d}+\frac{\varepsilon_{1}}{2}-\frac{2\varepsilon_{1}}{d}} 
\delta(s_{n})^{\frac{2+\varepsilon_{1}}{d}}
+
\delta(s_{n})
s_{n}^{\frac{d-4}{2}+\frac{\varepsilon_{1}}{4}}
+
s_{n}^{\frac{\varepsilon_{1}}{2}}
\to 0
\quad 
\mbox{as $n\to \infty$}
.
\end{split}
\end{equation}
Thus, we have proved \eqref{22/2/23/17:18} and therefore 
 the claim \eqref{22/2/22/16:31} is true. 


\noindent 
{\bf Proof of Claim 7.}~We shall prove \eqref{21/4/9/10:48}. 

It follows from $-\Delta+s_{n}+V =\{1+ V (-\Delta +s_{n})^{-1} \}
(-\Delta+ s_{n})$ that  \eqref{eq:4.30} is written as   
\begin{equation}\label{proof-eq-9}
\widetilde{z}_{n} 
=
(- \Delta + s_{n})^{-1}  
\{1+ V (- \Delta + s_{n})^{-1} \}^{-1} 
\widetilde{V}_{n} \widetilde{z}_{n}
,
\end{equation}
where 
\begin{equation}\label{21/1/6/11:29}
\begin{split}
\widetilde{V}_{n}
&:=
V 
+
\frac{d+2}{d-2}
Q_{n}^{\frac{4}{d-2}} 
+
\{\mu_{n}M_{\omega_{n}}\}^{-\frac{4}{d-2}}
g'(\mu_{n}M_{\omega_{n}}Q_{n}) 
\\[6pt]
&=
\frac{d+2}{d-2} 
\big\{ Q_{n}^{\frac{4}{d-2}} - W^{\frac{4}{d-2}} \big\} 
+
\{\mu_{n}M_{\omega_{n}}\}^{-\frac{4}{d-2}}
g'(\mu_{n}M_{\omega_{n}}Q_{n}) .
\end{split} 
\end{equation}
By \eqref{proof-eq-9}, and the adjoint operator of 
$\{1+ V (- \Delta + s_{n})^{-1} \}^{-1}$ on $L^{2}_{\rm{real}}(\mathbb{R}^{d})$ being $\{1+ (- \Delta + s_{n})^{-1}V \}^{-1}$,   
 we see that 
\begin{equation}\label{21/4/26/10:15}
\langle Q_{n}^{\perp}, \widetilde{z}_{n} \rangle
=
\langle 
\{1+ (- \Delta + s_{n})^{-1}V \}^{-1} 
(-\Delta+ s_{n})^{-1}Q_{n}^{\perp}, \widetilde{V}_{n} \widetilde{z}_{n}
\rangle
. 
\end{equation}
Furthermore, by \eqref{21/4/26/10:15}, H\"older's inequality, \eqref{18/11/11/15:50} in Proposition \ref{18/11/17/07:17}, \eqref{18/09/12/07:26} and $p_{1}+1>\frac{d}{d-2}$ (see Assumption \ref{21/12/23/14:45}), we see that
\begin{equation}\label{21/4/26/10:10}
\begin{split}
| \langle Q_{n}^{\perp}, \widetilde{z}_{n} \rangle |
&\le 
\| 
\{1+ (- \Delta + s_{n})^{-1}V \}^{-1} 
(-\Delta+ s_{n})^{-1}Q_{n}^{\perp}
\|_{L^{p_{1}+1}}
\| \widetilde{V}_{n} \widetilde{z}_{n} \|_{L^{\frac{p_{1}+1}{p_{1}}}}
\\[6pt]
&\lesssim
\| 
(-\Delta+ s_{n})^{-1}Q_{n}^{\perp}
\|_{L^{p_{1}+1}}
\| \widetilde{V}_{n} \widetilde{z}_{n} \|_{L^{\frac{p_{1}+1}{p_{1}}}}
.
\end{split}
\end{equation}
Consider the first factor on the right-hand side of \eqref{21/4/26/10:10}. 
 First, observe from \eqref{21/4/5/10:16}, $V\Lambda W =\Delta \Lambda W$, Lemma \ref{21/1/10/15:37} and \eqref{22/2/24/18:55} that 
\begin{equation}\label{22/2/24/18:5}
\begin{split}
|\tau_{n}|
&\lesssim 
| \langle  (-\Delta + s_{n})^{-1} Q_{n},V \Lambda W \rangle | 
= 
| \langle  \Delta Q_{n}, 
(-\Delta + s_{n})^{-1}  \Lambda W \rangle |
\\[6pt]
&\le 
\|\Delta  Q_{n}\|_{L^{1}} 
\| (-\Delta + s_{n})^{-1} \Lambda W \|_{L^{\infty}} 
\lesssim 
\delta(s_{n})^{-1} 
. 
\end{split} 
\end{equation}
Then, by \eqref{18/09/12/07:25}, \eqref{18/11/24/13:56} in Lemma \ref{18/11/05/10:29}, Lemma \ref{18/11/23/17:17} together with $p_{1}+1>\frac{d}{d-2}$, \eqref{22/2/24/18:5} and \eqref{19/01/01/12:18}, we see that 
\begin{equation}\label{21/3/31/10:1}
\begin{split}
&
\| (-\Delta +s_{n})^{-1} Q_{n}^{\perp} \|_{L^{p_{1}+1}} 
=
\| (-\Delta +s_{n})^{-1} \{Q_{n} -\tau_{n}V\Lambda W \} \|_{L^{p_{1}+1}} 
\\[6pt]
&\le 
\| (-\Delta +s_{n})^{-1} Q_{n}  \|_{L^{p_{1}+1}}
+
|\tau_{n}| 
\| (-\Delta +s_{n})^{-1} V \Lambda W  \|_{L^{p_{1}+1}} 
\\[6pt]
&\lesssim 
s_{n}^{\frac{d-2}{2}-\frac{d}{2(p_{1}+1)}-1}
\|Q_{n}\|_{L_{\rm{weak}}^{\frac{d}{d-2}}}
+
\delta(s_{n})^{-1} \| V \Lambda W  \|_{L^{\frac{d(p_{1}+1)}{d+2(p_{1}+1)}}}
\\[6pt]
&\lesssim
s_{n}^{\frac{d-2}{2}-\frac{d}{2(p_{1}+1)}-1}
+
\delta(s_{n})^{-1}
.
\end{split} 
\end{equation}
Consider the second factor on the right-hand side of \eqref{21/4/26/10:10}. 
 Let $0< \varepsilon <\frac{d}{2(p_{1}+1)}$. 
 Note that $W^{\frac{6-d}{d-2}}\in L^{\frac{p_{1}+1}{p_{1}-1}}(\mathbb{R}^{d})$  for $d=3,4$, and $\frac{d(p_{1}+1)}{d-2\varepsilon (p_{1}+1)}>\frac{d}{d-2}$. 
 Then, by the definition of $\widetilde{V}_{n}$ (see \eqref{21/1/6/11:29}), 
 the second inequality in \eqref{22/1/3/17:15},  H\"older's inequality,  \eqref{19/01/01/12:18}, \eqref{20/12/28/10:26}, \eqref{19/01/01/08:23} and $M_{\omega_{n}}^{p_{2}-\frac{d+2}{d-2}} 
 \sim t_{n} \sim \delta(s_{n})^{-1}s_{n}$ (see \eqref{22/1/3/11:23}, \eqref{22/1/3/14:5} and \eqref{21/4/28/11:48}), we see that   
\begin{equation}\label{21/1/6/11:47}
\begin{split} 
&\|   \widetilde{V}_{n} \widetilde{z}_{n}  
\|_{L^{\frac{p_{1}+1}{p_{1}}}} 
\lesssim 
\| ( W+|\zeta_{n}| )^{\frac{6-d}{d-2}} \zeta_{n} 
\widetilde{z}_{n}\|_{L^{\frac{p_{1}+1}{p_{1}}}} 
+ 
M_{\omega_{n}}^{p_{2}-\frac{d+2}{d-2}}  \| Q_{n}^{p_{1}-1}\widetilde{z}_{n} 
\|_{L^{\frac{p_{1}+1}{p_{1}}}} 
\\[6pt]
&\le
\| (W+|\zeta_{n}|)^{\frac{6-d}{d-2}} \|_{L^{\frac{p_{1}+1}{p_{1}-1}}}
\|\widetilde{z}_{n}\|_{L^{\frac{d(p_{1}+1)}{d-2\varepsilon (p_{1}+1)}}}
\| \zeta_{n}  \|_{L^{\frac{d}{2\varepsilon}}}
+
M_{\omega_{n}}^{p_{2}-\frac{d+2}{d-2}}
\| Q_{n}^{p_{1}-1} \|_{L^{\frac{p_{1}+1}{p_{1}-1}}}
\| \widetilde{z}_{n}\|_{L^{p_{1}+1}}
\\[6pt]
&\lesssim  
s_{n}^{\frac{d-2}{2}-\varepsilon}
+
\delta(s_{n})^{-1}s_{n}
,
\end{split} 
\end{equation}
where the implicit constants may depend on $\varepsilon$.   
 Putting \eqref{21/4/26/10:10}, \eqref{21/3/31/10:1} and \eqref{21/1/6/11:47} together, and taking a sufficiently small $\varepsilon$ depending on $d$ and $p_{1}$, we see from \eqref{19/01/27/16:58} that 
\begin{equation}\label{21/4/26/11:41}
\begin{split}
&
\delta(s_{n})
| \langle Q_{n}^{\perp}, \widetilde{z}_{n} \rangle |
\lesssim  
\delta(s_{n})
\big\{
s_{n}^{\frac{d-2}{2}-\frac{d}{2(p_{1}+1)}-1}
+
\delta(s_{n})^{-1}
\big\}
\big\{ 
s_{n}^{\frac{d-2}{2}-\varepsilon}
+
\delta(s_{n})^{-1}s_{n}
\big\}
\\[6pt]
&=
\big\{
\delta(s_{n})
s_{n}^{\frac{d-2}{2}-\frac{d}{2(p_{1}+1)}-1}
+
1 
\big\}
\big\{ 
s_{n}^{\frac{d-2}{2}-\varepsilon}
+
\delta(s_{n})^{-1}s_{n}
\big\}
\\[6pt]
&=
\delta(s_{n})
s_{n}^{d-2-\frac{d}{2(p_{1}+1)}-1-\varepsilon}
+
s_{n}^{\frac{d-2}{2}-\frac{d}{2(p_{1}+1)}}
+ 
s_{n}^{\frac{d-2}{2}-\varepsilon}
+
\delta(s_{n})^{-1}s_{n}
\to 0
\quad 
\mbox{as $n\to \infty$}.
\end{split} 
\end{equation}
Thus, we have proved \eqref{21/4/9/10:48}.


We have proved all of the claims and therefore completed the proof of {\it (i)} of Theorem \ref{20/8/17/11:39}.  
\end{proof}


\subsection{The kernel of $\boldsymbol{L_{\omega,+}}$}\label{21/10/9/13:18}
In this section, we prove the second claim {\it (ii)} of  Theorem \ref{20/8/17/11:39}. Since $\Phi_{\omega}$ is radially symmetric about $0$, we may regard it as a function of single variable on $[0,\infty)$; we also use the symbol $\Phi_{\omega}$ to denote such a single variable function. 

By \eqref{21/9/18/10:40}, what we need to prove is that 
\begin{equation}\label{21/10/10/11:58}
\operatorname{dim}{\operatorname{Ker}{L_{\omega,+}}}=d
.
\end{equation}
Although we can prove \eqref{21/10/10/11:58} by an argument similar to the proof of Theorem 0.3 in \cite{Kabeya-Tanaka}, we give a proof as the kernel contains  the components arising from the translation invariance of \eqref{eq:1.1}.

  For $d\ge 2$ and $k\ge 0$, let $N_{d,k}$ denote the dimension of the space of  spherical harmonics of degree $k$. It is known that $N_{d,0}=1$, $N_{d,1}=d$, and $N_{d,k}$ is finite for all $k\ge 2$. 
 Furthermore, for each $k\ge 0$,  let $Y_{k}^{1},\ldots, Y_{k}^{N_{d,k}}$ be mutually orthogonal spherical harmonics of degree $k$. Note that $Y_{k}^{m}$ and $Y_{\ell}^{n}$ are orthogonal if either $k \neq \ell$ or $m\neq n$, and 
\begin{equation}\label{21/9/17/11:48}
-\Delta_{S^{d-1}} Y_{k}^{m} = k(d+k-2) Y_{k}^{m} 
\qquad 
\mbox{for all $k\ge 0$ and $1\le m \le N_{d,k}$}
,  
\end{equation}
where $\Delta_{S^{d-1}}$ is the Laplace-Beltrami operator on 
 the sphere $S^{d-1}$.

Now, suppose for contradiction that \eqref{21/10/10/11:58} is false. Then, we can take a nontrivial function $f_{0} \in H^{2}(\mathbb{R}^{d})$ such that 
\begin{equation}\label{21/9/18/10:59}
L_{\omega,+}f_{0}=0,
\qquad 
\langle f_{0}, \partial_{j}\Phi_{\omega} \rangle =0
\quad 
\mbox{for all $1\le j \le d$}
.
\end{equation}
We will use the polar coordinates, namely $r=|x|$, $\sigma=\frac{x}{|x|}$ and $\sigma_{j}:=\frac{x_{j}}{|x|}$ ($1\le j \le d$) for $x\in \mathbb{R}^{d}\setminus \{0\}$. Note that the spherical harmonics are functions of $\sigma$. We may take $Y_{0}^{1}=1$ and $Y_{1}^{j} = \sigma_{j}$ for all $1\le j \le d$.
 Then, we may write $f_{0}$ as follows:
\begin{equation}\label{21/9/17/17:20}
f_{0}(x)
=
c_{0}(r) + \sum_{j=1}^{d} c_{1}^{j}(r) \sigma_{j}
+
\sum_{k=2}^{\infty} \sum_{m=1}^{N_{d,k}} c_{k}^{m}(r) Y_{k}^{m}(\sigma), 
\end{equation}
where 
\begin{equation}\label{21/9/18/11:20}
\begin{split} 
c_{0}(r)&:=\int_{S^{n-1}} f_{0}(r \sigma) d\sigma
, 
\\[6pt]  
c_{1}^{j}(r)
&:=\int_{S^{n-1}} \sigma_{j} f_{0}(r \sigma) d\sigma
\quad 
\mbox{for all $1\le j \le N_{d,1}=d$}
,
\\[6pt]
c_{k}^{m}(r)
&:=\int_{S^{n-1}} f_{0}(r \sigma) Y_{j}^{\ell}(\sigma) d\sigma
\quad 
\mbox{for all $k\ge 2$ and $1\le m \le N_{d,k}$}.
\end{split}
\end{equation}

We may write $L_{\omega,+}$ as    
\begin{equation}\label{21/10/9/13:14} 
L_{\omega,+}
=
-\frac{\partial^{2} }{\partial r^{2}}
-\frac{d-1}{r} \frac{\partial }{\partial r}
-
\frac{1}{r^{2}}\Delta_{S^{d-1}}
+
V_{\omega,+}  
\quad 
\mbox{with}~
V_{\omega,+}:=
\omega -
\frac{d+2}{d-2}\Phi_{\omega}^{\frac{4}{d-2}} 
-
g'(\Phi_{\omega})
. 
\end{equation}

Let $k\ge 0$ and $1\le m \le N_{d,k}$. 
 By $L_{\omega,+}f_{0}=0$ (see \eqref{21/9/18/10:59}), \eqref{21/9/17/17:20},  
 \eqref{21/10/9/13:14} and \eqref{21/9/17/11:48}, we see that 
\begin{equation}\label{21/9/18/11:7}
0
= 
\int_{S^{d-1}} L_{\omega,+}f_{0} (r\sigma)
Y_{k}^{m}(\sigma) 
\, d\sigma
=
A_{\omega,k} c_{k}^{m}(r) 
,
\end{equation}
where 
\begin{equation}\label{21/9/18/11:37}
A_{\omega,k}   
:=
-\frac{d^{2}}{dr^{2}} - \frac{d-1}{r} \frac{d}{dr} 
+\frac{k(d+k-2)}{r^{2}}
+
V_{\omega,+} 
.
\end{equation}
Note that $A_{\omega,0}=L_{\omega,+}|_{H_{\rm{rad}}^{2}(\mathbb{R}^{d})}$. 
 Hence, the nondegeneracy of $\Phi_{\omega}$ in $H_{\rm{rad}}^{1}(\mathbb{R}^{d})$ (see the claim {\it (i)} of Theorem \ref{20/8/17/11:39}) shows that $c_{0}$ must be trivial. 
 
 We shall show that $c_{k}^{m}$ are trivial for all $k\ge 1$ and $1\le m \le N_{d,k}$. To this end, put $\gamma_{k}^{m}:=r^{\frac{d-1}{2}}c_{k}^{m}$. Then, we may write \eqref{21/9/18/11:7} as 
\begin{equation}\label{21/9/27/14:9}
\frac{d^{2}\gamma_{k}^{m}}{dr^{2}}
=
\Big\{
\frac{(d-1)(d-3)}{4r^{2}}
+
\frac{k(d+k-2)}{r^{2}} 
+
V_{\omega,+}(r)
\Big\} \gamma_{k}^{m} 
.
\end{equation}
Note that $\gamma_{k}^{m}(0)=0$, and $\gamma_{k}^{m}(r)\to 0$ exponentially as $r\to \infty$ (see, e.g., the proof of Lemma 2 of \cite{Berestycki-Lions}). Furthermore, the standard theory of linear ODEs shows that the zeros of $\gamma_{k}^{m}$ are isolated.

In order to derive a contradiction, it suffices to show that $\gamma_{k}^{m}$ is trivial for all $k\ge 1$.  Let $k\ge 1$ and $1\le m \le N_{d,k}$. Furthermore, let $a_{k,m}>0$ be the second zero of $\gamma_{k}^{m}$; The first zero is $0$, and we regard $a_{k,m}=\infty$ if $\gamma_{k,m}$ has no zero in $(0,\infty)$. 

Suppose for contradiction that $\gamma_{k}^{m}$ is nontrivial. 
 Then, considering $-\gamma_{k}^{m}$ instead of $\gamma_{k}^{m}$ if necessary, 
 we may assume that 
\begin{equation}\label{21/9/24/6:1}
\gamma_{k}^{m}(r)>0~~\mbox{for all $0< r <a_{k}^{m}$},
\quad 
\frac{d\gamma_{k}^{m}}{dr}(0)\ge 0 ,
\quad 
\frac{d\gamma_{k}^{m}}{dr}(a_{k}^{m})< 0
~~\mbox{if $a_{k}^{m} \neq \infty$}
.
\end{equation}  
Note that $\dfrac{d\gamma_{k}^{m}}{dr}(a_{k}^{m})\neq 0$, as otherwise $\gamma_{k}^{m}$ becomes trivial.  

Since $\Phi_{\omega}$ is strictly decreasing as a function of $|x|$,  we see that 
\begin{equation}\label{21/9/24/7:31}
\dfrac{d\Phi_{\omega}}{dr}(r)<0 \quad 
\mbox{for all $r>0$}
.  
\end{equation}
Observe that 
\begin{equation}\label{21/9/24/7:17}
\partial_{j}\Phi_{\omega}(x) 
=
\sigma_{j} \frac{d\Phi_{\omega}}{dr}(r)
.
\end{equation}
By \eqref{21/9/18/10:40}, \eqref{21/10/9/13:14}, \eqref{21/9/24/7:17} and \eqref{21/9/17/11:48}, we see that 
\begin{equation}\label{21/9/26/18:25}
A_{\omega,1}\frac{d\Phi_{\omega}}{dr} 
=
-\frac{d^{3}\Phi_{\omega}}{dr^{3}}
-\frac{d-1}{r}\frac{d^{2}\Phi_{\omega}}{dr^{2}}
+\frac{d-1}{r^{2}}\frac{d\Phi_{\omega}}{dr}
+V_{\omega,+}\frac{d\Phi_{\omega}}{dr}
=0. 
\end{equation}
By \eqref{21/9/27/14:9}, integration by parts and \eqref{21/9/26/18:25}, 
 we see that 
\begin{equation}\label{21/10/10/16:23}
\begin{split} 
&
\int_{0}^{a_{k}^{m}} 
\Big\{
\frac{(d-1)(d-3)}{4r^{2}}
+
\frac{k(d+k-2)}{r^{2}} 
+
V_{\omega,+}(r)
\Big\} \gamma_{k}^{m} \frac{d\Phi_{\omega}}{dr} r^{\frac{d-1}{2}} \,dr
\\[6pt]
&=
\xi(a_{k}^{m}) -
\int_{0}^{a_{k}^{m}}
\frac{d \gamma_{m}^{k}}{dr}
\frac{d^{2}\Phi_{\omega}}{dr^{2}} r^{\frac{d-1}{2}} \,dr 
-
\int_{0}^{a_{k}^{m}}\frac{d \gamma_{m}^{k}}{dr}
\frac{d \Phi_{\omega}}{dr} \frac{d-1}{2} r^{\frac{d-3}{2}} \,dr 
\\[6pt]
&=
\xi(a_{k}^{m}) +  
\int_{0}^{a_{k}^{m}} \gamma_{m}^{k} 
\frac{d^{3}\Phi_{\omega}}{dr^{3}} r^{\frac{d-1}{2}} \,dr 
+
\int_{0}^{a_{k}^{m}}
\gamma_{m}^{k} 
\frac{d^{2}\Phi_{\omega}}{dr^{2}} \frac{d-1}{2r}r^{\frac{d-1}{2}} \,dr 
\\[6pt]
&\quad +
\int_{0}^{a_{k}^{m}} \gamma_{m}^{k}
\frac{d^{2} \Phi_{\omega}}{dr^{2}} \frac{d-1}{2r} r^{\frac{d-1}{2}} \,dr 
+
\int_{0}^{a_{k}^{m}} \gamma_{m}^{k}
\frac{d \Phi_{\omega}}{dr} \frac{(d-1)(d-3)}{4r^{2}} r^{\frac{d-1}{2}} \,dr 
\\[6pt]
&=
\xi(a_{k}^{m})
 +
\int_{0}^{a_{k}^{m}} 
\Big\{ 
\frac{d-1}{r^{2}}
+
\frac{(d-1)(d-3)}{4r^{2}}
+
V_{\omega,+} 
\Big\} 
\gamma_{m}^{k}
\frac{d \Phi_{\omega}}{dr} r^{\frac{d-1}{2}} \,dr 
,
\end{split}
\end{equation}
where 
\begin{equation}\label{21/10/11/9:51}
\xi(a_{k}^{m}):=\Big( 
\frac{d \gamma_{m}^{k}}{dr}(r) 
\frac{d\Phi_{\omega}}{dr}(r) r^{\frac{d-1}{2}}\Big)\Big|_{r=a_{k}^{m}}
.
\end{equation}
Hence, we see that 
\begin{equation}\label{21/10/11/10:3}
\int_{0}^{a_{k}^{m}} 
\frac{k(d+k-2)- (d-1)}{r^{2}} 
\gamma_{m}^{k}
\frac{d \Phi_{\omega}}{dr} r^{\frac{d-1}{2}} \,dr 
=
\xi(a_{k}^{m})
.
\end{equation}
Note that 
\begin{equation}\label{21/10/11/11}
\frac{k(d+k-2)- (d-1)}{r^{2}}
> 0
\quad 
\mbox{for all $d\ge 3$ and $k\ge 2$}
.
\end{equation}
Observe from \eqref{21/9/24/6:1} and \eqref{21/9/24/7:31} that 
\begin{align}
\label{21/10/11/9:55}
&
\gamma_{k}^{m}(r) \frac{d\Phi_{\omega}}{dr}(r) r^{\frac{d-1}{2}} <0
\quad 
\mbox{for all $0<r <a_{k}^{m}$},
\\[6pt]
\label{21/10/11/11:40} 
& \xi(a_{k}^{m})>0~~\mbox{if $a_{k}^{m}\neq \infty$},
\qquad 
\xi(a_{k}^{m})=0~~\mbox{if $a_{k}^{m}=\infty$}
.
\end{align}
Assume either $k\ge 1$ and $a_{k}^{m}\neq \infty$, or else $k\ge 2$. 
 Then, \eqref{21/10/11/10:3} together with \eqref{21/10/11/11}, \eqref{21/10/11/9:55} and \eqref{21/10/11/11:40} shows a contradiction. It remains the case where $k= 1$ and $a_{1}^{m}= \infty$ for some $1\le m \le d$. Note that $a_{1}^{m}=\infty$ implies that $\gamma_{1}^{m}>0$ on $(0,\infty)$. 
 Moreover, by $\langle f_{0}, \partial_{m}\Phi_{\omega} \rangle =0$ (see \eqref{21/9/18/10:59}), \eqref{21/9/24/7:17} and the definition of $c_{1}^{m}(r)$ (see \eqref{21/9/18/11:20}), we see that 
\begin{equation}\label{22/3/22/10:9}
0=
\int_{0}^{\infty} 
\int_{S^{d-1}}f_{0}(r\sigma)\sigma_{m} \frac{d\Phi_{\omega}}{dr}(r)r^{d-1}\,d\sigma dr
=
\int_{0}^{\infty} c_{1}^{m}(r) \frac{d\Phi_{\omega}}{dr}(r)r^{d-1}\,dr
. 
\end{equation}
  Then, by $\gamma_{1}^{m}$ being positive, $\dfrac{d\Phi_{\omega}}{dr}<0$ (see \eqref{21/9/24/7:31}),  $\gamma_{1}^{m}(r):=r^{\frac{d-1}{2}}c_{1}^{m}(r)$ and 
\eqref{22/3/22/10:9},  we see that  
\begin{equation}\label{21/9/23/17:43}
0> \int_{0}^{\infty} 
\gamma_{1}^{m}(r) 
\frac{d\Phi_{\omega}}{dr}(r) r^{\frac{d-1}{2}}\,dr
=
0
.
\end{equation}
This is a contradiction. Thus, we have proved that $c_{k}^{m}$ is trivial for all $k\ge 0$ and $1\le m \le N_{d,k}$, which implies that $f_{0}$ is trivial. 
However, this contradicts $f_{0}$ being nontrivial. 
 Thus, the second claim {\it (ii)} of Theorem \ref{21/9/17/9:1} must be true. 


\section{Proof of Theorem \ref{21/9/17/9:1}}\label{21/10/7/10:45}

In this section, we give a proof of Theorem \ref{21/9/17/9:1}:
\begin{proof}[Proof of Theorem \ref{21/9/17/9:1}]
Let us begin with the proof of the claim {\it (i)}. 
 Suppose for contradiction that there exists a real-valued function $f_{*} \in H^{1}(\mathbb{R}^{d})$ such that 
\begin{align}
\label{21/9/14/15:30}
&
\mathcal{N}_{\omega}'(\Phi_{\omega})f_{*}=0,  
\\[6pt]
\label{21/9/14/15:31}
&
B_{\omega}(f_{*}, f_{*})=-1.
\end{align} 
Then, define the function $Z\colon \mathbb{R}^{2}\to \mathbb{R}$ by  
\begin{equation}
Z(a,b):=\Phi_{\omega}+ af_{*} + b \Phi_{\omega}
.
\end{equation} 
Note that $\mathcal{N}_{\omega}(Z(0,0))=\mathcal{N}_{\omega}(\Phi_{\omega})=0$.  Moreover, by \eqref{21/9/15/16:23}, $\Phi_{\omega}$ being a positive solution to \eqref{eq:1.1}, and the condition \eqref{22/3/6/12:26}, we see that 
\begin{equation}\label{21/9/16/13:36}
\begin{split}
&\frac{\partial}{\partial b} 
\mathcal{N}_{\omega}(Z(a, b))\Big|_{(a,b)=(0,0)}
=
\mathcal{N}_{\omega}'(\Phi_{\omega}) 
\Phi_{\omega} 
=
\langle L_{\omega,+}\Phi_{\omega}, \Phi_{\omega} \rangle
\\[6pt]
&=
-\frac{4}{d-2}\langle 
\Phi_{\omega}^{\frac{d+2}{d-2}}, \Phi_{\omega} 
\rangle
-
\langle g'(\Phi_{\omega})\Phi_{\omega}-g(\Phi_{\omega}), \Phi_{\omega} 
\rangle 
<0
.
\end{split} 
\end{equation} 
Hence, the implicit function theorem together with \eqref{21/9/14/15:30} shows 
 that there exists $a_{0}>0$ 
 and a $C^{2}$-function $h \colon (-a_{0},a_{0})\to \mathbb{R}$ such that 
\begin{align}
\label{21/9/16/13:51}
&
\mathcal{N}_{\omega}(Z(a,h(a))) =0 
\qquad 
\mbox{for all $a\in (-a_{0},a_{0})$},
\\[6pt]
\label{21/9/16/14:20}
& 
h(0)=0,
\qquad 
\frac{dh}{da}(0)=-
\frac{\mathcal{N}_{\omega}'(\Phi_{\omega})f_{*}}
{\mathcal{N}_{\omega}'(\Phi_{\omega})\Phi_{\omega}}
=0.
\end{align}
Since $\Phi_{\omega}$ is a ground state to \eqref{eq:1.1}, 
 \eqref{21/9/16/13:51} shows that 
\begin{equation}\label{21/9/16/13:55}
\mathcal{S}_{\omega}(Z(a,h(a)))
\ge 
\mathcal{S}_{\omega}(\Phi_{\omega})
\quad 
\mbox{for all $a\in (-a_{0},a_{0})$}
.
\end{equation}
Observe from $\mathcal{S}_{\omega}'(\Phi_{\omega})=0$ (see \eqref{21/12/23/15:59}) that 
\begin{equation}\label{21/9/16/15:2}
\frac{d}{da}
\mathcal{S}_{\omega}(Z(a,h(a)))\Big|_{a=0}
=
\mathcal{S}_{\omega}'(\Phi_{\omega})\{f_{*}+\frac{dh}{da}(0)\Phi_{\omega}\}
=0
.
\end{equation}
Furthermore, by \eqref{21/9/16/14:20}, \eqref{21/9/16/14:33} and \eqref{21/9/14/15:31}, 
 we see  that  
\begin{equation}\label{21/9/16/13:56}
\frac{d^{2}}{da^{2}}
\mathcal{S}_{\omega}(Z(a,h(a)))\Big|_{a=0}
=
\big[
\mathcal{S}_{\omega}''(\Phi_{\omega})\{f_{*}+\frac{dh}{da}(0)\Phi_{\omega}\}
\big] 
\{f_{*}+\frac{dh}{da}(0)\Phi_{\omega}\}
=-1
,
\end{equation}
which together with $\mathcal{S}_{\omega}$ and $h$ being $C^{2}$ implies that 
there exists $0< a_{1} <a_{0}$ such that  
\begin{equation}\label{21/9/16/16:35}
\frac{d^{2}}{da^{2}}
\mathcal{S}_{\omega}(Z(a,h(a)))
<0
\quad 
\mbox{for all $a \in (-a_{1},a_{1})$}
.
\end{equation}
Thus, Taylor's expansion of $\mathcal{S}_{\omega}(Z(a,h(a)))$ around $a=0$ together with \eqref{21/9/16/15:2} and \eqref{21/9/16/16:35} shows that 
\begin{equation}\label{21/9/16/15:15}
\mathcal{S}_{\omega}(Z(a,h(a))) - \mathcal{S}_{\omega}(\Phi_{\omega})
=
\frac{a^{2}}{2} 
\frac{d^{2}}{da^{2}}
\mathcal{S}_{\omega}(Z(a,h(a))) \Big|_{a=\theta a}
< 0
\quad 
\mbox{for all $a \in (-a_{1},a_{1})$}
,
\end{equation}
where $\theta\in (0,1)$ is some constant. However, this contradicts \eqref{21/9/16/13:55}. Thus, we have proved the claim {\it (i)}. 

We move on to the proof of the claim {\it (ii)}. It follows from $\sigma_{\rm{disc}}(L_{\omega,+})\subset (-\infty, \omega)$ (see \eqref{22/3/22/11:4}) and $\langle L_{\omega,+}\Phi_{\omega}, \Phi_{\omega} \rangle <0$ (see \eqref{21/9/16/13:36}) that $L_{\omega,+}$ has at least one negative eigenvalue. Let $E_{1}<0$ be the first eigenvalue, and let $E_{2}$ be the second one counting multiplicity
 (possibly, $E_{1}= E_{2}$). 
 Since $0\in \sigma_{\rm{disc}}(L_{\omega,+})$ (see \eqref{21/9/18/10:40}), we see  that $E_{2}\le 0$. Moreover, the min-max principle (see, e.g., Theorem XIII.2 of \cite{Reed-Simon-IV}) together with \eqref{21/9/15/16:23} and the first claim {\it (i)} shows that 
\begin{equation}\label{21/10/15/9:51} 
E_{2} 
= 
\sup_{\phi \in H^{1}(\mathbb{R}^{d})} 
\inf_{{u\in H^{1}(\mathbb{R}^{d}), \|u\|_{L^{2}}=1}\atop{\langle u, \phi \rangle=0}} 
B_{\omega}(u, u)
\ge 
\inf_{{u\in H^{1}(\mathbb{R}^{d}), \|u\|_{L^{2}}=1}\atop{\langle u, L_{\omega,+}\Phi_{\omega} \rangle=0}} 
B_{\omega}(u, u)
\ge 0. 
\end{equation}
Thus, we see that $E_{2}=0$ and the claim {\it (ii)} is true. 
\end{proof}


\bibliographystyle{plain}


\end{document}